\documentclass[a4paper,12pt]{article} %
\usepackage{etex}
\usepackage[latin1]{inputenc}
\usepackage[T1]{fontenc}
\usepackage[english,french]{babel}
\let\savedDelta\Delta
\usepackage{amsmath}
\let\Delta\savedDelta
\usepackage{amsmath,amssymb,amsthm}
\usepackage{amsfonts,pifont}
\usepackage{mathrsfs}
\usepackage{stmaryrd}
\usepackage{graphicx}
\usepackage{epsfig}
\usepackage{color} 
\usepackage[all,dvips]{xy}
\usepackage{textcomp}
\usepackage{pstricks,pst-tree}
\usepackage{pst-3d}
\usepackage{pstricks-add}
\usepackage{pst-text}
\usepackage{pst-plot}
\usepackage{pst-3dplot}
\usepackage{pst-solides3d,pst-fractal}
\usepackage{tikz}
\usepackage{tikz-qtree}
\usepackage{comment}
\usepackage{multicol}
\usepackage{cancel}

\usepackage[papersize={21cm,29.7cm},  left=2cm,right=2cm,twoside,  top=2.5cm, bottom=2.5cm,  dvips,verbose]{geometry}

\title{Ecalle's averages, Rota-Baxter algebras and the construction of moulds}
\author{Emmanuel Vieillard-Baron\\
}

\newcommand{\R}{\mathbb{R}}

\newcommand{\p}[1]{\left(#1\right)}
\newcommand{\cro}[1]{\left[#1\right]}

\newcommand{\ens}[1]{\left\{#1\right\}}

\newcommand{\bas}[1]{\check{#1}}

\newcommand{\F}{\mathbb{F}}

\newcommand{\N}{\mathbb{N}}

\renewcommand{\R}{\mathbb{R}}

\newcommand{\RPs}{\R_+^*}

\newcommand{\C}{\mathbb{C}}

\newcommand{\intere}[2]{\left\llbracket #1,#2\right\rrbracket}

\newcommand{\integ}[4]{\int_{#1}^{#2} #3 \text{d} #4}
\newcommand{\dd}[1]{\mathrm{ \;d}#1}

\def\point{\bullet}

\newcommand{\ENDOM}{\text{\textbf{ENDOM}}\p{\C\cro{\cro{u}}}}
\newcommand{\Lin}{\text{Lin}}

\newcommand{\mots}{{\Omega}^\point}

\newcommand{\m}{\underline m}
\newcommand{\n}{\underline n}

\newcommand{\w}{{\underline{\omega}}}

\newcommand{\sh}{\text{\textbf{sh}}}

\newcommand{\pscal}[1]{\left< #1\right>}

\newcommand{\motsarbo}{{\Omega^{\point<}}}

\newcommand{\K}{\text{K}}
\newcommand{\id}{\text id}

\newcommand{\eps}{\epsilon}
\newcommand{\rev}{\text{\textbf{rev}}}

\newcommand{\im}{\text{Im }}

\def\norme#1{\Vert#1\Vert}

\def\abs#1{\left|#1\right|}
\def\l#1{\text{l}\p{#1}}

\newcommand{\dspappli}[5]{
#1:  \left\{
    \begin{array}{ccc}
      #2 & \longrightarrow & #3 \\
      #4 & \mapsto & #5
    \end{array}
  \right.
  }

\newlength{\textlarg}

\def\QSym{{\it QSym}}          
\def\WQSym{{\bf WQSym}}        

\def\C{{\mathbb C}}

\def\binomial#1#2{\left(\,\begin{matrix}#1 \cr #2\end{matrix}\,\right)}

\def\<{\langle}
\def\>{\rangle}

\def\F{{\bf F}}         

\def\N{{\bf N}}

\excludecomment{commentaire}
 \newcommand{\ancien}[1]{}
 \newcommand{\ancienm}[1]{}

 \newcommand{\refere}[1]{}

\newcommand{\pointa}{\bullet}

\newcommand{\arbrea}[1]{\scalebox{0.5}{\begin{tikzpicture}
\tikzset{grow'=up,level distance=30pt}
\tikzset{execute at begin node=\strut}
\Tree [.$\pointa #1$ ]
\end{tikzpicture}}}

\newcommand{\arbreabA}[2]{\scalebox{0.5}{\begin{tikzpicture}
\tikzset{grow'=up,level distance=30pt}
\tikzset{grow'=up}
\tikzset{execute at begin node=\strut}
\Tree [.$\pointa #1$ [.$#2$ ] ]
\end{tikzpicture}}}

\newcommand{\arbreabB}[2]{\scalebox{0.5}{\begin{tikzpicture}
\tikzset{grow'=up,level distance=30pt}
\tikzset{execute at begin node=\strut}
\Tree [ .$\pointa #1$   ]
\begin{scope}[xshift=+7mm]
\Tree   [ .$\pointa #2$   ]
\end{scope}
\end{tikzpicture}}}

\newcommand{\arbreabcA}[3]{\scalebox{0.5}{\begin{tikzpicture}
\tikzset{grow'=up,level distance=30pt}
\tikzset{execute at begin node=\strut}

\Tree [.$\pointa #1$
 [.$#2$  ] [.$#3$ ]  ]
\end{tikzpicture}}}

\newcommand{\arbreabcB}[3]{\scalebox{0.5}{\begin{tikzpicture}
\tikzset{grow'=up,level distance=30pt}
\tikzset{execute at begin node=\strut}
\Tree [.$\pointa #1$ [.$#2$  ]  ]
\begin{scope}[xshift=+7mm]
\Tree [.$\pointa #3$ ]  
\end{scope}
\end{tikzpicture}}}

\newcommand{\arbreabcdA}[4]{\scalebox{0.5}{\begin{tikzpicture}
\tikzset{grow'=up,level distance=30pt}
\tikzset{execute at begin node=\strut}
\Tree [.$\pointa #1$ [.$#2$  ] [.$#3$  ]  [.$#4$  ]  ]
\end{tikzpicture}}}

\newcommand{\arbreabcdB}[4]{\scalebox{0.5}{\begin{tikzpicture}
\tikzset{grow'=up,level distance=30pt}
\tikzset{execute at begin node=\strut}
\Tree [.$\pointa #1$ [.$#2$ $#3$  ]  [.$#4$  ]  ]
\end{tikzpicture}}}

\newtheorem{theorem}{Theorem}[section]
\newtheorem{lemma}[theorem]{Lemma}
\newtheorem{proposition}[theorem]{Proposition}

\renewenvironment{proof}[1][Proof]{\begin{trivlist}
\item[\hskip \labelsep {\bfseries #1}]}{$\boxempty$\end{trivlist}}

\newtheorem{notation}{Notation}[subsection]
\newtheorem{example}{Example}[subsection]

\newtheorem{definition}{Definition}[subsection]

\theoremstyle{plain}
\theoremstyle{remark}
\newtheorem{remark}{Remark}[subsection]

\renewcommand{\Lin}[1]{\text{G}_{#1}}

\newcommand{\adds}{+}

\renewcommand{\id}{\text{id}}

\renewcommand{\ENDOM}{\text{\textbf{ENDOM}}}

\def\sh{
\setlength{\unitlength}{.5 pt}
\begin{picture}(40,20)
\put(10,2){\line(1,0){20}}
\put(10,2){\line(0,1){10}}
\put(20,2){\line(0,1){10}}
\put(30,2){\line(0,1){10}}
\end{picture}}

\newcommand{\Hcsh}{\text{\textbf{Hqsh}}}
\newcommand{\Hsh}{\text{\textbf{Hsh}}}

\newcommand{\HCK}{\text{\textbf{Hck}}}

\newcommand{\phidiff}{\phi_{\text{diffusion}}}

\newcommand{\phiorga}{\phi_{\text{organic}}}

\renewcommand{\gg}{>}
\renewcommand{\N}{{{\mathbb{N}}}}

\renewcommand{\R}{{{\mathbb{R}}}}
\renewcommand{\RPs}{{{\R_+^*}}}
\renewcommand{\C}{{{\mathbb{C}}}}
\renewcommand{\K}{{{\mathbb{K}}}}

\def\qshu{\joinrel{\!\scriptstyle\amalg\hskip -3.1pt\amalg}\,\hskip -8pt\hbox{-}\hskip 5pt}

\usepackage{version}
\newif\ifshow
\showtrue

\ifshow
  \includeversion{WithArbo}
\excludeversion{WithoutArbo}
\else
\includeversion{WithoutArbo}
  \excludeversion{WithArbo}
\fi

\newif\ifdiff
\difffalse

\ifdiff
  \includeversion{WithDiff}
\excludeversion{WithoutDiff}
\else
\includeversion{WithoutDiff}
  \excludeversion{WithDiff}
\fi

\begin{document}
\begin{otherlanguage}{english}

\maketitle

\begin{small}
 \begin{center}
Université de Bourgogne \\
Institut de Mathématiques de Bourgogne\\ 
9, avenue Alain Savary\\
B.P. 47 870, 21078 Dijon, France\\
email:\,\textit{emmanuel.vieillard-baron@math.cnrs.fr}
 \end{center}

\end{small}

\textbf{Classification~:} 05C05, 0E99, 16T30, 30B40.

\textbf{Keywords~:} quasi-shuffle algebra, Rota-Baxter algebra, character, Atkinson recursion, tree, mould, arborification, Hopf algebra.

\abstract{Rota-Baxter algebras and Atkinson's method are powerful tools for the factorization of characters on Hopf algebras. The theory of real resummation discovered by J. Ecalle and known as \textit{well-behaved averages theory} can be reformulated in terms of character factorization. The aim of this article is to explain how Atkinson recursion provides an alternative way to retrieve characters already discovered by Ecalle.}

\tableofcontents

\end{otherlanguage}

\section{Introduction}

Since the pioneering work of Baxter in fluctuation theory (\cite{Ba}) in the sixties, followed by those of Cartier (\cite{Ca}) and Rota (\cite{Ro}, \cite{RoSi}) (ten years after) and finally by the one of Atkinson \cite{At}, Rota-Baxter algebras have become  a powerful frame for the factorization of characters  on Hopf algebras.  There are many applications; for example, the treatment by Connes and Kreimer of the renormalization problem in pQFT can be understood through this mechanism (\cite{CoKr}).

In the middle of the eighties and in a totally distinct domain, J. Ecalle found an original approach of local analytical dynamic based on a calculus making use of new objects, \textit{moulds} (\cite{Ec1}, \cite{Ec2}, \cite{Sa1}, \cite{Sa2}, \cite{Sa3}, \cite{Sa4},...). It appeared in the last decade that Ecalle's formalism may be naturally translated into the language of Hopf algebra. Thus, some particular moulds called \textit{symmetrel} may  be identified with characters on the quasi-shuffle algebra (\cite{Ho2}). This was the starting point of a fruitful series of exchanges between the world of algebraic combinatorics and the one of local dynamic (\cite{Ch}, \cite{ChHiNoTh}, \cite{Cr}, \cite{FaFoMa}, \cite{MeNoTh}, \cite{MeNoTh1}, \cite{Sc}, \cite{NoPaSaTh} ....).

The \textit{resummation problem} consists in finding, for a given divergent power series  $\tilde \phi\in\C[[x]]$, a function $\phi$ defined and analytic on a sectoral neighbourhood of $0$ in the complex plane that is asymptotic to $\tilde \phi$. Since the $19$th century, numerous works have been devoted to this problem. The seminal work by Malgrange, Ramis and Sibuya has led to the multisummability theory and independently the one of  Ecalle to the  \textit{accelero-summability theory} (\cite{MalRa}, \cite{MarRa}, \cite{RaSi}).

The particular case of \textit{real resummation} consists in summing a real and divergent power series into an analytic function defined on a sectoral neighbourhood of $0$ bisected by one of the two real directions which takes  real values along real direction. It is a very delicate problem and its solution by Ecalle, known as \textit{well-behaved averages theory} (\cite{Ec2}, \cite{Me2}, \cite{Me4}, \cite{Vi1}) is equivalent to finding some  character of the quasi-shuffle algebra $\Hcsh$  satisfying some very restrictive conditions. Ecalle has discovered several families of solutions of this problem. Thus he has introduced for this purpose some \textit{averages} that he called \textit{diffusion induced} and \textit{organic} (\cite{Ec1}). This was a spectacular solution to a very hard problem. But  their origin remains mysterious and the proofs given by Ecalle were written in a very concise style. 

The problem for constructing averages can be split into two parts~: an algebraic and an analytic one.  Let us explain its algebraic component.

Before doing it, we have to introduce the map $\iota:\Hcsh\to\R$ and defined  for any word $\w\in\Hcsh$ by $$\iota(\w)=\begin{cases} 1&\text{ if } \l{\w}\leq 1 \\ 0&\text{ otherwise } \end{cases} $$ where $\l{\w}$ is the length of the word $\w$. It is a character of $\Hcsh$.

We denote by $*$ the convolution product between characters of $\Hcsh$ and we consider the algebra morphism  
$\rev:\Hcsh\to\Hcsh$  defined for any word  $\w=(\omega_1,\hdots,\omega_n)\in\Hcsh$ by $\rev\p{\w}=\p{\omega_n,\hdots,\omega_1}$ and extended on the whole of $\Hcsh$ by linearity.

The \textit{algebraic problem of averages} consists in finding a character $\phi:\Hcsh\to\C$ obeying  $$\phi^{*-1}*\psi=\iota \text{ where } \psi^{*-1}=\overline{\phi}\circ \rev.$$

Solving it is then equivalent to factorizing the character $\iota$ into two  characters $\phi$ and $\psi$.  What is particular here is that these characters  are linked by the relation $\psi^{*-1}=\overline{\phi}\circ \rev$ and so the \textbf{factorization is constrained}.

The aim of this paper is to show that Atkinson's method on specified Rota-Baxter algebras provides a very natural frame to obtain such a factorization. But as we will explain, the hard constraints in the factorization will be satisfied only for some appropriate choices of the working Rota-Baxter algebra. We will exhibit three Rota-Baxter algebras solving this problem and allowing to obtain the characters already discovered by Ecalle.


Beyond providing an effective way  to compute  characters satisfying the algebraic problem of averages, the Rota-Baxter formalism allows to easily solve its \textit{analytic} component. 
\begin{WithoutArbo}A next paper will be dedicated to this subject and it consists to compute the \textit{contracted arborificationtion} of the character $\phi$. This problem can easily be translated into a factorization one no more on the quasi-shuffle algebra but on the Connes-Kreimer algebra. 
\end{WithoutArbo}
\begin{WithArbo}
It amounts to calculating the \textit{contracted arborification} of characters satisfying the algebraic problem of averages and to obtain some good bounds for it. \textit{Arborification theory} is a very important element of  Ecalle's work (\cite{Ec1},\cite{EcVa}). The contracted arborified of a given character on the quasi-shuffle algebra  is obtained when composing it  by an algebra morphism $\alpha:\HCK\to\Hcsh$ defined from the Connes-Kreimer Hopf algebra $\HCK$ into the quasi-shuffle one $\Hcsh$. This morphism, known as \textit{arborification morphism}, is now a well known object but its computation requires some difficult combinatorial considerations (\cite{FaMe}, \cite{Vi1}, \cite{Vi2}). When applying Atkinson's factorization no more on the character $\iota$ defined on the the quasi-shuffle algebra but directly on its contracted arborification $\iota^<=\iota\circ \alpha$ defined on the Connes-Kreimer algebra, recursive procedure allows then to compute the contracted arborification of diffusion induced and organic characters without knowing the explicit form of the arborification morphism.

The \textit{analytic problem of averages} consists then in proving the geometrical growth of the sequence $(\phi^{<}(\w^<))_{\w^<\in\HCK}$, i.e. in proving the existence of $a,b\in\RPs$ satisfying for any forest $\w^<\in\HCK$~:
$$\abs{\phi^{<}\p{\w^<}}\leq ab^{\norme{\w^<}} $$ where  $\phi^<=\phi\circ \alpha$ is a character solving the averages algebraic problem and  where $\norme{\w^<}$ represents the sum of all the decorations of the nodes of the forest $\w^<$. We explain that it can be reformulated in terms of the \og growth \fg~ of the used Rota-Baxter operator.   
\end{WithArbo}
      

The paper is organized as follows. We begin by some basic reminders about Rota-Baxter algebras, Atkinson's factorization and the quasi-shuffle Hopf algebra. We then establish our main Theorem \ref{theo_wba} which gives a criterium allowing to detect Rota-Baxter algebras on which the above mentioned constrained factorization is possible, i.e Rota-Baxter algebras solving algebraic and analytic averages problem. 

\begin{WithArbo}
After some basic facts about Connes-Kreimer Hopf algebra and the arborification morphism, we establish Theorem  \ref{theo_algo_arborification} delivering an algorithm for the computation of the contracted arborification. We then establish Theorem  \ref{theo_wbaa} which allows to recognize Rota-Baxter algebras solving the analytic problem of averages. 
\end{WithArbo}

%
We then apply our main theorem to rediscover Ecalle's characters defining diffusion induced and organic averages. We propose two Rota-Baxter algebras on which it is possible to obtain the character associated to organic average. 
\begin{WithArbo}
We prove that these Rota-Baxter algebras solve the analytic problem of averages too and  we obtain closed formula for the contracted arborification of these two characters using the algorithm explained in Theorem \ref{theo_algo_arborification}. 
\end{WithArbo}

Discovering new well behaved averages consists now in finding new Rota-Baxter algebras satisfying the hypotheses of Definition \ref{def_wbRBa} and Definition \ref{def_wbaRB}. It seems that except the one of section \ref{section_diffusion_induced_averages} (which is example 1.1.16 of \cite{Gu}), no Rota-Baxter algebra among the most classical ones (as the other Rota-Baxter algebras illustrating \cite{Gu}) is in position to  verify these hypotheses.

Let us finally mention an article of  Menous-Novelli-Thibon \cite{MeNoTh} dedicated to several constructions of characters on combinatorial Hopf algebras based on ideas related to well-behaved theory and more precisely to properties of diffusion induced averages family. In particular, the authors propose  an expression of the character $\phi$ associated with the diffusion induced average in terms of iterated Rota-Baxter operators. The presentation of the combinatoric of averages in terms of a character factorization problem and its solution using Atkinson methods, however, is specific to our work. Furthermore, we will explain in a next article that it is possible to obtain the main relation of \cite{MeNoTh} using Atkinson recursion on the Hopf algebra of  \textit{packed words} $\WQSym$ which is a non-commutative lift of $\QSym$, the Hopf algebra of word quasi-symmetric functions.  This approach leads to recover the first representation of averages given by Ecalle as a system of complex scalar weights indexed by words on a given alphabet and satisfying some so called \textit{autocoherence relations} for which  we give an algebraic interpretation. 

A preliminary version of this work was exposed in a workshop at the Scuola Normale Superiore Pisa 
in the beginning of april 2013.  We thank the members  of the project ANR CARMA 12-BS01-0017 which has permitted the organization  of this workshop. Finally, we thank Frédéric Fauvet for his constant support and his advices and Dominique Manchon for his final careful reading of the paper. 



\section{Rota-Baxter algebras and Atkinson recursion}\label{subsection_reminders_quasi_shuffles_algebra}

We recall now some classical definitions and results about Rota-Baxter algebras and the Atkinson recursion. The reader can consult \cite{EbMaPa} or \cite{Gu} for proofs and details.

\begin{definition}
Let us consider an associative algebra $A$ and an endomorphism $R\in\ENDOM \p{A}$. The pair $\p{A,R}$ is said to be a \textbf{Rota-Baxter algebra of weight $\theta\in\K$} if $R$ satisfies the Rota-Baxter relation~:
\begin{equation}R(x)R(y)=R\p{R(x)y+xR(y)+\theta xy} .\label{Rota_Baxter_relation}\end{equation} 
\end{definition}

\begin{notation}We set $\tilde R=-\theta \id_A-R$.
\end{notation}


\begin{theorem}[Atkinson recursion]
 Let $\p{B,R}$ be an associative unital Rota-Baxter algebra. Consider $a\in B$ and set
$$F=\sum_{n\in\N} t^n \p{Ra}^{\cro{n}} \text{ and } G=\sum_{n\in\N} t^n \p{\tilde R a}^{\ens{n}}  $$ with inductively defined~: 
\begin{itemize}
                                                                                                                                \item $\p{Ra}^{\cro{0}}:=1_B$,  $\p{Ra}^{\cro{1}}:=R(a)$ and $\p{Ra}^{\cro{n+1}}:=R\p{\p{Ra}^{\cro{n}}a}$.
 \item $\p{Ra}^{\ens{0}}:=1_B$,  $\p{Ra}^{\ens{1}}:=R(a)$ and $\p{Ra}^{\ens{n+1}}:=R\p{a\p{Ra}^{\ens{n}}}$.
                                                                                                                               \end{itemize}
Then $F$ and $G$ solve the recursions 
\begin{equation}F=1_B+t R(F a) \label{equ_gamma_minus}\end{equation}
\begin{equation}G=1_B+t \tilde R(aG) \label{equ_gamma_plus} \end{equation}

in $B[[t]]$ and we have the factorization~:
$$ F\p{1_B+at\theta}G=1_A .$$
This factorization is unique for an idempotent Rota-Baxter map $R$. 
\end{theorem}

%

We consider now a graded connected commutative Hopf algebra $\p{H,\mu,\eta_H,\delta,\eps}$ and $\p{A,R}$ an associative unital Rota-Baxter algebra of weight $\theta$.

We denote by $\Lin{H,A}$ the set of algebra morphisms from $H$ to $A$. We turn this space into an algebra by considering the convolution product $\star$ defined for any $f,g\in\Lin{H,A}$ by, using Sweedler's notation, 
$$f\star g ( h) = f\p{h^{\p{1}}}.g\p{h^{\p{2}}}$$
for any $h\in H$. The unit is given by the map $e=\eta_A\circ \eps$.

Considering the map $\mathcal R:\Lin{H,A} \to \Lin{H,A}$ given for any $f\in\Lin{H,A}$ by $\mathcal R (f)=R\circ f$, we turn $\p{\Lin{H,A},\mathcal R}$ into a Rota-Baxter algebra of weight $\theta$ as well.

\begin{remark}
Let us observe that for any algebra morphism $\Theta:A\to \C$,  $\Theta\circ f$ is a character on the Hopf algebra $H$ for any $f\in \Lin{H,A}$. 
\end{remark}
 
%


\section{The quasi-shuffle Hopf algebra}\label{section_quasi_shuffles_algebra}

\subsection{Some reminders about the quasi-shuffle Hopf algebra}\label{subsection_remeinders_quasi_shuffles_algebra}

We denote by $\mots$ the set of words with letters in  a set $\Omega$. For the countable semi-group $\p{\Omega,\adds}$ with $\Omega=\ens{\omega_1,\omega_2,\hdots}$ and $\omega_i+\omega_j=\omega_{i+j}$, we consider  the linear span $\pscal{\mots}$ of $\mots$ and we inductively define on it the \textit{quasi-shuffle product} by, for any $a\m\footnote{where $a\m$ denotes the concatenation of the letter $a$ with the word $\m$},b\n\in\mots$~:

$$a\m\qshu b\n= a \p{\m\qshu b\n}+b\p{a\m\qshu \n}+\p{a\adds b} \m\qshu \n.$$

For example, one has~:
\begin{eqnarray*}\p{a,b}\sh\p{\alpha,\beta}=\p{\alpha,a,b,\beta}+\p{\alpha,a,\beta,
b} + \p{\alpha,\beta,a,b}+\p{a,\alpha,b,\beta}+\p{a,\alpha,\beta,b}+\p{a,b,\alpha,\beta}+\\
\p{\alpha\adds a,b,\beta}+ \p{\alpha,a,b\adds\beta}+ \p{\alpha\adds a,\beta,b}+
\p{\alpha+a,\beta+b}+ 
\p{\alpha,a,\beta+b}+ \p{\alpha,\beta\adds a,b}+\\
\p{a+\alpha,b,\beta}+\p{a\adds\alpha,b\adds\beta}+ \p{a,b\adds\alpha,\beta}
.\end{eqnarray*}

We have a natural grading on $\pscal{\mots}$ defined by $\omega_i\mapsto i$ and a natural coproduct given by deconcatination~: 
\begin{equation} \Delta\p{\w}=\sum_{\w^1.\w^2=\w}\w^1\otimes \w^2.\end{equation}

The unit is the application $\eta_H:1_\K\to \emptyset$ and the counit is $\eps:\w\to \begin{cases} 1&\text{ if } \w=\emptyset\\ 0 &\text{otherwise}\end{cases}$. 

We then consider the graded completed bialgebra $\Hcsh$ of $\pscal{\mots}$. It is a Hopf algebra called Hopf algebra of quasi-shuffle. We invite the reader to consult \cite{Ho2} for more details about the quasi-shuffle Hopf algebra construction. 


\subsection{Application of Atkinson recursion on the quasi-shuffle Hopf algebra}\label{subsection_quasi_shuffle_Hopf_algebra}

\begin{remark}\label{remark_rigidity_qs}
 Let us consider an associative, commutative and unital Rota-Baxter algebra $(A,R)$ and  $\phi:\Hcsh\to A$ an algebra morphism which vanishes on words of length strictly more than one. Then, for any $\omega_1,\omega_2\in\Omega$, because of the definition of the quasi-shuffle product, we must have
$$\phi\p{\omega_1\sh\omega_2}=\phi\p{\omega_1}\phi\p{\omega_2}=\phi\p{\omega_1+\omega_2}.$$ and then, if $\Omega=\N^*$, it comes, for any $\omega\in\Omega$~:
$$\phi\p{\omega} =\p{\phi\p{1}}^\omega.$$ 
\end{remark}

In relation with this remark, for $\gamma \in A $, we introduce the algebra morphism $\mathcal I\in\Lin{\Hcsh,A}$ defined by 
\begin{equation}
 \mathcal I\p{\w}= \begin{cases}
                    1_A & \text{ if } \w=\emptyset\\ 
                   \gamma^{{\omega}} &\text{ if } \l{\w}=1 \quad (\w=\omega)\\
                   0&\text{ otherwise }
                  \end{cases}
\end{equation} and we consider an algebra morphism $\Theta:A \to \C$ in a such way that $\Theta(\gamma)=1$. Then one has $\Theta\circ \mathcal I=\iota$.

By applying Atkinson's recursion (see \cite{EbMaPa}) in the Rota-Baxter algebra $\p{\Lin{\Hcsh,A}[[t]],\mathcal R}$, there exists a unique pair $\p{F,G}\in\p{\Lin{\Hcsh,A}[[t]]}^2$ such that, with\footnote{It is important to notice for what follows that $a\p{\w}=0$ for sequences $\w$ of length equal or more than $2$ (and for sequence of length $0$, i.e. for the empty sequence). } $a=-\p{\mathcal I-e}$, 
\begin{equation}
 e+at=F^{-1}G^{-1}
\end{equation}

and so for $t=-1$, one has~: 
\begin{equation}
 \mathcal I=F^{-1}_{|t=-1} G^{-1}_{|t=-1}.
\end{equation}

As $\Theta$ is an algebra morphism, $\Theta\circ F_{|t=-1}$ and $\Theta\circ G^{-1}_{|t=-1}$ define characters on $\Hcsh$ respectively denoted $\phi$ and $\psi$ and the previous relation becomes
 \begin{equation}
\iota=\phi^{*-1} * \psi.
\end{equation}

But using the properties of $\mathcal I$ and these of the quasi-shuffle Hopf algebra,  we can be more precise.

\begin{lemma}\label{lemma_1}
 One has ~: $$\phi(\point):=\Theta\circ F_{|t=-1}(\point)=\p{-1}^{\l{\point}}\p{\mathcal R a}^{\cro{\l{\point}}}(\point)$$
           $$\psi^{*-1}\circ\rev(\point):=\p{\Theta\circ G_{|t=-1}}\circ\rev(\point)=\p{-1}^{\l{\point}}\p{\tilde{\mathcal R }a}^{\cro{\l{\point}}}(\point)$$
\end{lemma}

\begin{proof}
For any $\w=\p{\omega_1,\hdots,\omega_n}\in \Hcsh$, one has 
\begin{eqnarray*}
 \phi\p{\w}=\Theta\circ F_{|t=-1}\p{\w}&=&\sum_{k\in\N} \p{-1}^k \p{\mathcal R a}^{\cro{k}}\p{\w}\\
&=& \p{-1}^n \p{ \mathcal R a}^{\cro{n}}\p{\w}
\end{eqnarray*}
because of the definition of $a$ and of the convolution product in $\Lin{\Hcsh,A} $. 
One also has
\begin{eqnarray*}
\psi^{*-1}\circ\rev\p{\w}&=&{\Theta\circ G_{|t=-1}}\circ\rev\p{\w}\\
&=&\p{-1}^n \Theta\circ\p{\tilde{\mathcal R} a}^{\ens{n}}\p{\rev\p{\w}}\\
&=&
\p{-1}^n \Theta\p{\tilde{ R}\p{ a\p{\omega_n} \tilde{ R}\p{ a\p{\omega_{n-1}} \tilde{ R}\p{a\p{\omega_{n-2}} \tilde{ R}\p{\hdots a_{\omega_{2}}\p{ \tilde{ R}\p{a\p{\omega_1}}}} }  } }}
\end{eqnarray*}
 and so, using the fact that $A$ is a commutative algebra
\begin{eqnarray*}{\Theta\circ\p{\tilde{ R} a}^{\ens{n}}}\circ\rev\p{\w}&=&
\Theta\p{{\tilde R}\p{ a\p{\omega_n} \tilde{ R}\p{ a\p{\omega_{n-1}} \tilde{ R}\p{a\p{\omega_{n-2}} \tilde{ R}\p{\hdots a_{\omega_2}\p{ \tilde{ R}\p{a\p{\omega_1}}}} }  } }}\\
&=&\Theta\p{\tilde{ R}\p{\tilde{ R}\p{\tilde{ R}\p{\hdots \tilde{ R}\p{ \tilde{ R}\p{a\p{\omega_1}}a\p{\omega_{2}}}}\hdots a\p{\omega_{r-1}}}a\p{\omega_r}}}\\
&=&\Theta\circ\p{\tilde{ R} a}^{\cro{n}}\p{\w}
\end{eqnarray*}
as needed.
\end{proof}

As a direct consequence of this lemma, one can claim that the character $\phi$ is a solution of the algebraic problem of well-behaved averages  if and only if $\overline{\phi(\point)}=\psi^{*-1}\circ\rev$, which will be the case if and only if 
\begin{equation}
 \overline{{\Theta\circ G_{|t=-1}}\circ\rev}=\Theta\circ F_{|t=-1}
\end{equation}

The point is that this equality can be obtained only for some good choice of the Rota-Baxter algebra $\p{A,R}$ and of the element $\gamma\in A$ as we will explain it in the next sub-section.

\subsection{The main theorem}\label{subsection_main_theorem}

\begin{definition}\label{def_wbRBa}
 A $6$-uple $\p{A,R,\theta,\sigma,\Theta,\gamma}$ is called an \textbf{average algebra} if and only if~:
\begin{enumerate}
 \item $A$ is a unitary algebra and $R:A\to A$ is an idempotent Rota-Baxter operator of weight $-1$. We denote by $\tilde R$ the complementary projector.\label{def_wbRBa_H1}
\item $\theta$  and $\sigma$ are two maps from $A$  to $A$ and $\gamma\in A$ such that  for any $b\in \im \tilde R$ and $n\in\N$, $\sigma\p{\gamma^n b}=\gamma^n\sigma\p{b}$. \label{def_wbRBa_H2}
\item $\Theta:A\to\C$ is an unital algebra morphism such that $\Theta(\gamma)=1$.\label{def_wbRBa_H0}
\item $\theta \circ\tilde R=R \circ\sigma \mod \ker \Theta$\label{def_wbRBa_H3}
\item $\sigma\p{\gamma}=\gamma$  \label{def_wbRBa_H4}
\item For any $b\in\im\tilde R$ and any $n\in\N$, $R\p{\gamma^n\tilde R\p{\sigma(b)}}=0$.\label{def_wbRBa_H5}
\end{enumerate}
\end{definition}


\begin{remark}{}\label{remark_simplification1}
If $\sigma$ is an algebra morphism, then Axiom \ref{def_wbRBa_H2} derives automatically from Axiom \ref{def_wbRBa_H5}. 
\end{remark}

\begin{lemma}\label{lemma_axiom5}     With the notations of Definition \ref{def_wbRBa}, one has for any $b\in \im\tilde R$ and any $n\in\N^*$~: 
$$R\p{\gamma^n\tilde R\p{\sigma(b)}}=R\p{\gamma^n\p{\sigma \circ \tilde R(b)-R\circ\sigma (b)}}.$$           
\end{lemma}           

\begin{proof}
For  $b\in \im\tilde R$ and any $n\in\N^*$, one has~:
\begin{align*}
R\p{\gamma^n\p{\sigma \circ \tilde R(b)-R\circ\sigma (b)}}
=R\Big(\gamma^n \p{\sigma(b) - R\circ {\sigma(b)}}\Big)
 =R\p{\gamma^n\tilde R\p{\sigma(b)}}
\end{align*}
because $b$ being an element of $\im \tilde R$, one has $\tilde R(b)=b$.
\end{proof}

\begin{remark}{}\label{remark_simplification2}
 A consequence of what precedes is that when $\theta=\sigma$, Axiom \ref{def_wbRBa_H5} is automatically fulfilled.
\end{remark}

The rest of this section is devoted to the proof of the main theorem~:

\begin{theorem}[Main theorem]\label{Rota-Baxter_caracterisation_of_wba}\label{theo_wba}
We consider an average algebra $\p{A,R,\theta,\sigma,\gamma,\Theta}$, an algebra morphism $\mathcal I:\Hcsh \to A$ such that \begin{equation}
 \mathcal I^{\w}= \begin{cases}
                    1_A & \text{ if } \w=\emptyset\\ 
                   \gamma^{{\omega}} &\text{ if } \l{\w}=1 \quad (\w=\omega)\\
                   0&\text{ otherwise }
                  \end{cases}
\end{equation}
and we set $a=-\p{\mathcal I-e}$.

We assume moreover that~: $$\forall n\in\N,\quad \overline{\Theta\p{\tilde{\mathcal Ra }}^{\cro{n}}}=\Theta\circ \theta\p{\p{\mathcal R a}^{\cro{n}}}.$$ Then 
the linear form  $\phi:=\Theta\circ F_{|t=-1}\in\Hcsh^*$  is  solution of the algebraic problem of averages, namely it is a character on $\Hcsh$ verifying~: 
\begin{align*}
 \phi^{*-1}*\psi=\iota \text{ and } \psi^{*-1}\circ\rev=\overline\phi.
\end{align*}
Moreover $\psi=\Theta\circ G_{|t=-1}^{-1}$.
\end{theorem}

%

We first prove the following lemma. With the same notations than in Theorem \ref{Rota-Baxter_caracterisation_of_wba}~:

\begin{lemma}\label{lemma_2}
 For any $n\in\N$, one has~: $$\theta\circ \p{\p{\tilde {\mathcal R}a}^{\cro{n}}}= \p{\mathcal Ra}^{\cro{n}}.  $$                           
\end{lemma}

\begin{proof}
We will prove that for any word $\w\in\Hcsh$, $\theta \p{\p{\tilde {\mathscr R}a}^{\cro{n}}\p{\w}}=\p{ {\mathscr R}a}^{\cro{n}}\p{\w}$ by an induction on the length of $\w$. Il $\l{\w}=0$ the formula is obvious. If $\l{\w}=1$ then one has 
\begin{equation*}
 \theta \p{\p{\tilde{\mathscr R}a}^{\cro{1}}\p{\w}}=\theta\p{\tilde R(a(\w))}=R\p{\sigma \p{a\p{\w}}}=R\p{a(\w)}
\end{equation*} because of Axiom \ref{def_wbRBa_H2} in Definition \ref{def_wbRBa} and of the definition of $a$.
We assume the property true for any word of length $\leq n-1$ and we prove it for a word $\w=\p{\omega_1,\hdots,\omega_{n}}$ of length $n$. To simplify the reading of the computation, we set $b=({\tilde{\mathcal R}a})^{\cro{n-1}}\p{\w'}$ where as usual $\w'=\p{\omega_1,\hdots,\omega_{n-1}}$. One has, modulo $\ker \Theta$~:
\begin{eqnarray}\label{conditions_wba}
\theta\p{\p{\tilde {\mathcal R}a}^{\cro{n}}\p{\w}}&=&   \theta\p{\tilde R\p{a(\omega_n) b }}     \label{c1}\\
&=&    R \p{\sigma \p{a(\omega_n)}\sigma\p{ b }}  \label{c2}\\
&=&    R \p{{a(\omega_n)} \sigma\p{ b }}  \label{c3}\\    
&=&     - R\p{a(\omega_n)}R\p{\sigma\p{ b }}
+ R\p{R\p{a(\omega_n)} {\sigma\p{ b }}  } 
+ R(a(\omega_n) \underbrace{R\p{\sigma\p{ b }} }_{=\theta \p{ b }} ) \label{c4} \\
&=& - R\cro{ a(\omega_n){R\p{\sigma\p{ b }} }}
+ R\cro{a(\omega_n) {\sigma\p{ b }}  }
+ R\cro{a(\omega_n){\theta \p{ b }}}\label{c5}\\
&=&     R\cro{a(\omega_n)\p{ \sigma\p{ b }- R\p{\sigma\p{ b }}   }}+ \p{Ra}^{[n]}\p{\w} \label{c6}\\
&=&     \underbrace{R\cro{a(\omega_n)\p{ \tilde R\p{\sigma( b)} }}   }_{=0}+ \p{Ra}^{[n]}\p{\w} \label{c7}\\
&=& \p{Ra}^{[n]}\p{\w}\label{c8}\end{eqnarray}
where~:
\begin{itemize}
 \item in $\p{\ref{c1}}$ we use the definition of $\p{\tilde {\mathcal R}a}^{\cro{n}}$,
\item $\p{\ref{c2}}$ is a consequence of Axioms \ref{def_wbRBa_H2} and \ref{def_wbRBa_H3} in Definition \ref{def_wbRBa},
\item $\p{\ref{c3}}$ is a consequence of Axiom \ref{def_wbRBa_H4} in Definition \ref{def_wbRBa} and of the definition of $a$,
\item $\p{\ref{c4}}$ is an application of the Rota-Baxter relation (\ref{Rota_Baxter_relation}) for Rota-Baxter operator of weight $-1$. We also apply the Axiom \ref{def_wbRBa_H1}  from which comes $\tilde R\circ \tilde R=\tilde R$ and so $$R\p{\sigma(b)}=R\p{\sigma\p{\p{\tilde{\mathcal R}a}^{\cro{n-1}}\p{\w'}}}=\theta \p{\tilde R \p{\p{\tilde{\mathcal R}a}^{\cro{n-1}}\p{\w'}}}  =\theta\p{\p{\tilde{\mathcal R}a}^{\cro{n-1}}\p{\w'}}=\theta\p{b}.$$
\item in $\p{\ref{c5}}$, we use once again Axiom \ref{def_wbRBa_H3} and the fact that $R$ is a projector whence $R\circ R=R$. Moreover it is a Rota-Baxtor operator, so $\im R$ is a sub-algebra of $A$ and   $$R\p{a(\omega_n)}R\p{\sigma\p{ b }}=a(\omega_n)R\p{\sigma\p{ b }}  \in \im R$$ which is fixed by $R$.
\item in $\p{\ref{c6}}$, we use the induction hypothesis~: 
$$R\cro{a(\omega_n){\theta \p{ b }}}=R\cro{a(\omega_n){\theta \p{\p{\tilde{\mathcal R}a}^{\cro{n-1}}\p{\w'}  }}}=R\cro{a(\omega_n){ \p{\mathcal Ra}^{\cro{n-1}}\p{\w'}  }}=\p{Ra}^{[n]}\p{\w}.$$
\item finally in $\p{\ref{c7}}$, we firstly use the fact that $\tilde R=\id -R$. Thus we can apply Axiom \ref{def_wbRBa_H5} in Definition \ref{def_wbRBa} which delivers $R\cro{R\p{a(\omega_n)}\mathcal  R\p{\sigma\p{ b }}   }=0$.
\end{itemize}
The formula is then proved by induction.
\end{proof}

We now prove the main theorem.

\begin{proof}[Proof of main theorem]

One already know from Lemma \ref{lemma_1} that for any word $\w\in\Hcsh$~: $$\phi(\w):=\Theta\circ F_{|t=-1}(\w)=\p{-1}^{\l{\w}}\p{\mathcal R a}^{\cro{\l{\w}}}$$
           $$ {\psi}\circ\rev(\w):={\Theta\circ G_{|t=-1}}\circ\rev(\w)=\p{-1}^{\l{\w}}\p{\tilde{\mathcal R }a}^{\cro{\l{\w}}}$$ and that $$\phi^{*-1}*\psi=\iota $$

 It remains to prove that $$\psi^{*-1}\circ \rev=\overline\phi.$$

 One applies the two previous lemmas. We know by Lemma \ref{lemma_1} that for any word $\w\in\Hcsh$,
$\phi(\w)=\p{-1}^{\l{\w}}\Theta\circ{\p{\mathcal R a}^{\cro{\l{w}}}}(\w)$
  and         ${\psi^{*-1}}\circ \rev(\w)=\p{-1}^{\l{\w}}\Theta\circ{\p{\tilde{\mathcal R }a}^{\cro{\l{\point}}}}(\w)$ so 
\begin{eqnarray*}
 \overline{{\psi^{*-1}}\circ \rev(\w)}&=&\overline{\p{-1}^{\l{\w}}\Theta\p{\p{\tilde{\mathcal R }a}^{\cro{\l{\w}}}} }(\w)\\
&=& \p{-1}^{\l{\w}}\Theta \circ \theta\p{\p{\tilde{\mathcal R }a}^{\cro{\l{\w}}}}(\w)\\
&=&\p{-1}^{\l{\w}}\Theta\p{\p{\mathcal R a}^{\cro{\l{\w}}}}\\
&=& \phi(\w)
\end{eqnarray*}
by application of Lemma \ref{lemma_2}.

\end{proof}

\begin{WithArbo}

\section{The Connes-Kreimer Hopf algebra}

\subsection{Some reminders about the Connes-Kreimer Hopf algebra}\label{subsection_Connes_Kreimer_Hopf_algebra}
We consider now for the Hopf algebra $H$ the one of Connes-Kreimer $\HCK$ with forests decorated by  nodes in the alphabet $\Omega$. The product is given by forests concatenations and the coproduct by \og degrafting \fg as we see below.

A rooted forest
is a non-planar graph with connected components that are rooted trees. An
$\Omega$-decorated forest is given by a couple $\p{\mathcal F,f}$ where $f$ is a
map from the vertices of $\mathcal F$ to $\Omega$. We naturally represent a
decorated forest by placing on each vertex of $\mathcal F$ its image by $f$. By
example~:
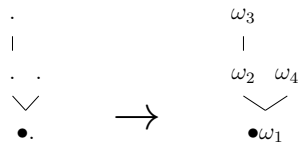
$$
\scalebox{1.5}{\arbreabcdB{.}{.}{.}{.} \quad
\raisebox{2.3mm}{$\rightarrow$} \quad
\arbreabcdB{\omega_1}{\omega_2}{\omega_3}{\omega_4}}
 $$
A decorated vertex belonging to a decorated forest $\p{\mathcal F,f}$  will be called a
node of $\p{\mathcal F,f}$.

We will denote by $\motsarbo$ the set of all $\Omega$-decorated
forests. In Ecalle's terminology, an $\Omega$-decorated forest is called an
arborescent
sequence and is denoted $\w^<=\p{\omega_1,\hdots,\omega_r}^<$ where
$\omega_1,\hdots,\omega_r\in\Omega$ are the nodes of $\w^<$.

An arborified sequence with 
only one rooted node, i.e. a decorated tree, is said to be \textit{irreducible}. If $\w$
admits several rooted
nodes $\omega_{i_1}$, $\hdots$, $\omega_{i_s}$  then the forest $\w^<$ is the
disjoint
union (namely the product)
of $s$ irreducible trees $\w^{<i_1}$, $\hdots$, $\w^{<i_s}$ with respective
rooted nodes   $\omega_{i_1}$, $\hdots$, $\omega_{i_s}$. We then write\footnote{Ecalle's notation is \mbox{$\w^{<}=\w^{<i_1}\oplus\hdots\oplus\w^{<i_s}$}.} $\w^{<}=\w^{<i_1}\hdots\w^{<i_s}$.

The length and the norm of an arborescent sequence is defined  in the same way as for a
totally ordered sequence of
$\Hcsh$.

We define, for any $\eta\in\Omega$, the operator $B_\eta^+:\HCK\to \HCK$.
For a forest $\w^<=\w^{1<}\hdots\w^{s<}$, $B_\eta^+\p{\w}$ is the rooted
tree with rooted node $\eta$ connected to the root of each  tree
$\w^{1<},\hdots,\w^{s<}$. We then define by induction the coproduct on $\HCK$. We set $\Delta\p{\emptyset}=\emptyset\otimes \emptyset$,
$\Delta\p{\w^{1<}\hdots\w^{s<}}=\Delta\p{\w^{1<}}\hdots
\Delta\p{\w^{s<}}$ and we assume that for a forest $\w^<$,
$$\Delta\p{B_\eta^+\p{\w^<}}=\emptyset \otimes B_\eta^+\p{\w^<} +
\p{B_\eta^+\otimes
\id}\circ \Delta\p{\w^<} .$$

One has for example~: 
 \begin{equation*}\Delta\p{\arbreabcA{1}{2}{3}}=\arbreabcA{1}{2}{3}\otimes \emptyset+\arbrea{1}\otimes \arbreabB{2}{3}+\arbreabA{1}{2}\otimes\arbrea{3}+\arbreabA{1}{3}\otimes\arbrea{2}+\emptyset\otimes\arbreabcA{1}{2}{3}.\end{equation*}

The coproduct of a forest $\w^<$ contains all tensor products of two forests $\w^{1<}$ and $\w^{2<}$ from which it it possible to obtain $\w^<$ when grafting  $\w^{2<}$ on a node of $\w^{1<}$. This is the reason why we use the name \og degrafting \fg for the coproduct.

Coming back to the previous subsection, for any $a\in\Omega$, we introduce the
linear map on $\Hcsh$
defined, for any $\w\in\mots$, by $L_{a}\p{\w}=a\w$.  Let us observe that one
has 
$$\Delta\p{L_{\omega}\p{\w}}=\emptyset\otimes
L_\omega\p{\w}+\p{L_{\omega}\otimes \id}\circ \Delta\p{\w} $$ and, following
\cite{Fo1}, there exists
an unique bigebra morphism $\alpha: \HCK\to \Hcsh$ such that
$$\alpha\circ
B_{\omega}^+=L_{\omega}\circ \alpha .$$ 

This morphism, called \textit{arborification morphism}, allows to associate to each linear form $\rho$ on $\Hcsh$, i.e.
a mould, a linear form $\rho\circ \alpha$ on $\HCK$, i.e. an
 \textit{arborescent mould}.    

Here are some examples~:

\begin{example}$\,$
 \begin{itemize}
\item $\alpha\p{\emptyset^<}=\emptyset$.
  \item
$\alpha\p{\raisebox{-1.5mm}{\scalebox{1.5}{\arbrea{\eta}}}}=\alpha\p{B^+_\eta\p{
\emptyset}}=L_\eta\p{\alpha\p{\emptyset}}=\eta$.
\item
$\alpha\p{\raisebox{-1.5mm}{\scalebox{1.5}{\arbreabA{\eta}{\beta}}}}=\alpha\p{
B^+_\eta\p{\raisebox{-1.5mm}{\scalebox{1.5}{\arbrea{\beta}}}}}=L_\eta\p{\alpha\p
{\raisebox{-1.5mm}{\scalebox{1.5}{\arbrea{\beta}}}}}=\p{\eta,\beta}$.
\item
$\alpha\p{\raisebox{-2.5mm}{\scalebox{1.5}{\arbreabB{\eta}{\beta}}}}=\alpha\p{
\raisebox{-1.5mm}{\scalebox{1.5}{\arbrea{\eta}}} }\qshu \alpha\p{
\raisebox{-1.5mm}{\scalebox{1.5}{\arbrea{\beta}}}
}=\p{\eta,\beta}+\p{\beta,\eta}+\eta\adds \beta$
\item
$\alpha\p{\raisebox{-3.5mm}{\scalebox{1.5}{\arbreabcA{\eta}{\beta}{\gamma}}}}
=\alpha\p{B^+_\eta\p{\raisebox{-2.5mm}{\scalebox{1.5}{\arbreabB{\beta}{\gamma}}}
}}=L_\eta\p{\alpha\p{\raisebox{-2.5mm}{\scalebox{1.5}{\arbreabB{\beta}{\gamma}}}
}}=L_\eta\p{\p{\beta,\gamma}+\p{\gamma,\beta}+\beta\adds
\gamma}=\p{\eta,\beta,\gamma}+\p{\eta,\gamma,\beta}+\p{\eta,\beta\adds \gamma}$

 \end{itemize}

\end{example}

But as we will explain now, the arborification of the character arising in this work can be obtained without having an explicit formula for $\alpha$. Nonetheless, such a formula can be computed (see for example \cite{Vi1}).

Instead of applying Atkinson's factorization to the map $\mathcal I$ on the Rota-Baxter algebra $\Lin{\Hcsh,A}$, we will apply it on  $\Lin{\HCK,A}$ and to the map $\mathcal I\circ \alpha$. 

\subsection{Atkinson recursion on the Connes-Kreimer Hopf algebra}

%
When doing that, we obtain two series 
$$F^<=\sum_{n\in\N} t^n \p{\mathscr Ra^<}^{\cro{n}} \text{ and } G^<=\sum_{n\in\N} t^n \p{\tilde{\mathscr R} a^<}^{\ens{n}}  $$ 
such that $$F^{<-1}G^{<-1}=1+a^< t\theta $$ with $a^<=-\p{\mathcal I\circ \alpha-e\circ \alpha}$.  One evidently has~: $F^{<}_{|t=-1}\p{\w^<}=F_{|t=-1}\circ\alpha\p{\w^<}$, $G^{<}_{|t=-1}\p{\w^<}=G_{|t=-1}\circ\alpha\p{\w^<}$.

And then, when composing them by the application $\Theta$, one obtains a character $\phi^<=\Theta \circ F^{<}_{|t=-1}=\Theta\circ F_{|t=-1}\circ\alpha$ known as the \textit{contracted arborification} of the character $ \phi=\Theta\circ F_{|t=-1}$.

The order of compositions in the computation of $\p{\mathscr R a}^{\cro{n}}$ makes it impossible to obtain a closed formula for the contracted arborified. But when reversing these compositions,  the strength of Rota-Baxter formalism allows to obtain such a formula without  having to explicit the arborification morphism $\alpha$.  In other words, the tools around Atkinson factorization enable us to deal with characters arborification without knowing how arborification works. Before explaining this, we need to introduce the algebra morphism $H:\Hcsh\to A$ given by $H=F_{|t=-1}\circ \rev.$ One can easily verify that $$H= \sum_{n\in\N} \p{-1}^n \p{\mathscr R a}^{\ens{n}}.$$ The reason why we introduce such an algebra morphism will appear in equation (\ref{equation_R_sur_irred}). The morphism $\Theta\circ H\circ \alpha$ is a well known object in Ecalle's work, it is the \textit{antiarborified} of the character $\phi$, see \cite{Ec1} or \cite{Vi2} for more details and we denote it  $\phi^\gg$.

For  $H^<= H\circ \alpha$, one has~:

\begin{theorem}\label{theo_algo_arborification} 
For a forest $\w^<\in\HCK$,
\begin{itemize}
 \item If $\w^<=\w^{1<}\hdots\w^{r<}$ where $\w^{i<}$ are irreducible for any $i\in\intere{1}{r}$  then 
\begin{equation}H^<\p{\w^{1<}\hdots\w^{r<}}=H^<\p{\w^{1<}}\hdots H^<\p{\w^{r<}}\label{equation_R_sur_red}\end{equation}
\item If $\w^<=\scalebox{1.5}{\raisebox{-2mm}{\arbreabcdA{\omega_1}{\w^{1<}}{\vdots}{\w^{l<}}}}\in \HCK$ is an irreducible arborified sequence of length $r$ then~:
\begin{equation}H^<\p{\w^<}= R\p{a\p{\omega_1}H^<\p{'\w^<} } \label{equation_R_sur_irred}\end{equation}
  where $'\w^<=\w^{1<}\hdots\w^{l<}$. 
\end{itemize}
\end{theorem}

\begin{proof}$\,$
\begin{itemize}
 \item The map $H^<$ being an algebra morphism from $\HCK$ into $A$, first formula is a direct consequence of the fact that the product in $\HCK$ is given by concatenation of forests.
\item For the second one, one has~:
\begin{eqnarray*}
 H^<\p{\w^<}&=&\sum  \p{-1}^n \p{\mathscr Ra}^{\ens{n}} \circ \alpha\p{\p{\w^<}}\\
&=& \sum_{n=0}^{+\infty}  \p{-1}^n \p{\mathscr Ra}^{\ens{n}} \circ  \alpha\p{B_{\omega_1}^+\p{\w^{1<}\hdots\w^{l<}}}\\
&=& \sum_{n=0}^{+\infty}  \p{-1}^n \p{\mathscr Ra}^{\ens{n}} \p{  L_{\omega_1}\p{\alpha\p{\w^{1<}\hdots\w^{l<}}}}\\
&=& \sum_{n=0}^{+\infty}  \p{-1}^n R\p{ a\p{\omega_1}\p{\mathscr Ra}^{\ens{n-1}}\circ\alpha\p{'\w^<} }\\
&=& R\p{ a\p{\omega_1}\p{\sum  \p{-1}^n \p{\mathscr Ra}^{\ens{n-1}}\circ\alpha\p{'\w^<} }}\\
&=& R\p{a\p{\omega_1} H^<\p{'\w^<} }
\end{eqnarray*}
 Let us observe that the previous sums are finite because words contained in the expansion $\alpha\p{\w^<}$ are of length $\leq \l{\w^<}$ and so $\p{\mathscr Ra}^{\ens{n}}\p{\w^<}=0$ if $\l{\w^<}<n$.
\end{itemize}

\end{proof}

\begin{definition}
 We say that a character $\kappa:\HCK\to\C$ has a \textbf{geometrical growth} if there exist $a,b\in\RPs$ such that for any forest $\w^<\in\motsarbo$~:
 $$\abs{\kappa\p{\w^<}} \leq a b^{\norme{\w}}.$$
\end{definition}

So the analytic problem of averages consists, for $\phi$ solution of the algebraic problem of averages,  in proving the geometrical growth of $\phi^<=\phi\circ\alpha$ . 

We will use  the following result of \cite{Vi2}~:

\begin{proposition}
 For a given  character $\vartheta:\HCK\to\C$, the character $\vartheta\circ\rev\circ\alpha:\HCK\to\C$ has a geometrical growth if and only if the character $\vartheta\circ \alpha:\HCK\to\C$ also has a geometrical growth. 
\end{proposition}

So solving the analytic problem of averages is equivalent to prove that $H^<=\phi\circ\rev\circ\alpha$ has a geometrical growth.


In order to do this, we need to control the \og growth \fg~ of the Rota-Baxter operator $R$. It is the motivation of next definition. 

\begin{definition} \label{def_wbaRB}
 We say that the seven-uple  $\p{A,R,\theta,\sigma,\Theta,\gamma,v}$ is a  \textbf{well-behaved Rota-Baxter algebra} if and only if $\p{A,R,\theta,\sigma,\Theta,\gamma}$ is an average algebra and  $v:A\to A$ is a multiplicative morphism  such that there exists $k>0$ verifying for any $m\in\N$ and $b\in \im R$~: 
\begin{enumerate}
\item \label{axiom_wba1} $ \abs{\Theta\circ {b}} \leq \abs{\Theta\circ v\p{{ b}}} $; 
\item \label{axiom_wba2} $\abs{\Theta\circ v\p{ R\p{\gamma^m b }}} \leq k^m \abs{\Theta\circ v\p { b}}$.
\end{enumerate}
\end{definition}

We then have~:

\begin{theorem}\label{theo_wbaa}
 Let us consider a well-behaved Rota-Baxter algebra $\p{A,R,\theta,\sigma,\Theta,\gamma,v}$. Then the character $\phi$ obtained with Atkinson recursion in Theorem \ref{theo_wba} solves the algebraic problem and the analytic problem of averages. 
\end{theorem}

\begin{proof}
We already know by Theorem \ref{theo_wba} that $\phi$ solves the algebraic problem of averages. We still have to prove that $\phi^<=\phi\circ\alpha$ has a geometrical growth, which is equivalent to prove the geometrical growth of $\phi^{\gg}=\Theta\circ H^{<}$.


We prove by induction on the length of $\w^<$ that $\abs{\Theta\circ v\circ H^{<}(\w^<)}\leq k^{\norme{\w^<}}$. 
\begin{itemize}
\item If $\l{\w^<}=0$, i.e. $\w^<=\emptyset$,  then ${\Theta\circ v\circ H^{<}(\w^<)}=1$. So  property is true for rank $0$. 
\item If $\l{\w^<}=1$, i.e if $\w^<=\omega_1$ with $\omega_1\in\mots$ then by Theorem \ref{theo_algo_arborification} and by Axiom \ref{axiom_wba2} of Definition \ref{def_wbaRB}, we know there exists $k'\in\RPs$ such that~:
\begin{align*}
\abs{\Theta\circ v\circ H^{<}(\w^<)}&=\abs{\Theta\circ v\circ R(a(\omega_1)}\leq k'^{\omega_1}\abs{\Theta\circ v(1)}=k'^{\omega_1}
\end{align*}
because $1\in\im R$.
\item We then set $k=\max(1,k')$.  We assume the inequality true for forests of length $\leq n$ and we prove it for a forest $\w^<\in\mots$ of length $n+1$.

 If $\w^<=\w^{1<}\hdots\w^{s<}$ is a product of trees $\w^{i<}$ for $i\in\intere{1}{s}$ then by application of the induction hypothesis on the trees $\w^1,\hdots,\w^s$ that are all of length $\leq n$~:
\begin{align*}
\abs{\Theta\circ v\circ H^{<}\p{\w^{1<} \hdots\w^{s<}}}&=   \abs{\Theta \circ v\circ H^{<} \p{\w^{1<}}}\hdots\abs{\Theta \circ v\circ H^{<} \p{\w^{s}}} \\
&\leq k^{\norme{\w^{1<}}} \hdots k^{\norme{\w^{s<}}}\\
&= k^{\norme{\w^<}} 
\end{align*}

If $\w^{<}=\scalebox{1.5}{\raisebox{-2mm}{\arbreabcdA{\omega_1}{\w^{1<}}{\vdots}{\w^{l<}}}}$ is irreducible then using induction hypothesis, Axiom \ref{axiom_wba2} of Definition \ref{def_wbaRB} and Theorem \ref{theo_algo_arborification}, it comes~:
\begin{align*}
\abs{\Theta \circ v\circ H^{<} \p{\w^<}}&= \abs{\Theta\circ v\circ{R\p{a\p{\omega_1}H^{<}\p{'\w^<} }}} \\
&\leq k^{\omega_1} \abs{\Theta\circ v\circ R\circ {H^{<}\p{'\w^<}}}\\
&\leq k^{\omega_1}\abs{\Theta\circ v\circ H^{<} \p{\w^<}}\\
&\leq k^{\omega_1}k^{\norme{\w'}}=k^{\norme{\w}}
\end{align*}
because $H^{\gg}(\w^<)\in\im R$. Thus, the property follows by a simple induction.

We then have, by Axiom \ref{axiom_wba1} of Definition \ref{def_wbaRB}, for any $\w^<\in\mots$~:
\begin{align*}
 \abs{\phi^{\gg}\p{\w^<}}&=\abs{\Theta\circ H^{<}\p{\w^<}}\leq \abs{\Theta\circ v\circ H^{<}\p{\w^<}}\leq k^{\norme{\w^<}}
\end{align*}

which proves that $\phi^{\gg}$ has an exponential growth and then that $\phi$ is solution of the analytical problem of averages.
\end{itemize}
\end{proof}

\begin{remark}\label{remark_restriction_axioms_wba}
 It is not necessary to verify inequalities of axioms \ref{axiom_wba1} and \ref{axiom_wba2} in definition \ref{def_wbaRB} for all elements $b\in\im R$ but just for elements $b$ of the form $ \p{\mathscr Ra}^{\ens{n}} \p{\w^<}$ for any $n\in\N$ and $\w\in\mots$.
\end{remark}

\end{WithArbo}

\section{Examples}\label{section_examples}

\subsection{Diffusion induced averages}\label{section_diffusion_induced_averages}

Diffusion induced averages were introduced in \cite{Ec1}, \cite{Me2} and studied in the framework of combinatorial Hopf algebras in \cite{MeNoTh}.

We first consider an analytic function $\beta\p{y}$ on $\R$. For any
$\omega\in\Omega$, we consider $g_\omega\p{y}=e^{-\omega\beta\p{y}}  $ and its
Fourier transform $f_\omega\p{x}=\frac{1}{2\pi}\int_\R
g_\omega\p{y}e^{ixy}\text{d}y$.  We obtain easily that~:
$$g_{\omega_1}\p{y}g_{\omega_2}\p{y}=g_{\omega_1+\omega_2}\p{y} \text{ and }
\p{f_{\underline \omega_1} * f_{\underline \omega_1}}\p{x}=\int_{\R}
f_{\underline \omega_1}\p{x} f_{\underline
\omega_2}\p{x-x_1}\text{d}x_1=f_{\underline \omega_1+\underline \omega_2}\p{x}$$
We assume that $\beta$ is chosen in a such way that for any
$\omega\in\Omega$, $\int_{\R} \abs{f_{\omega}\p{x}}\text{d}x=1$.

We easily verify that if $\beta$ is even then $\overline{f_{\omega}}(x)=f_{\omega}(-x)=f_{\omega}(x)$ for any $x\in\R$.


We assume that $A$ is the linear space of bounded integrable functions on $\R$ whose restrictions to $\R_+^*$ and $\R_-^*$ are continuous. The functions $f_\w$ previously introduced are elements of this space.

We fit this space with the convolution product $\star$ and, by adjunction of the formal symbol $\delta$, that can be interpreted as the Dirac distribution at $0$, $A$ become an unitary algebra. We define the operator $R:A\mapsto A$ by  $R\p{f}(x)=f(x)\sigma_+\p{x}$ where $\sigma_+$ is the indicator function of $\RPs$.   The pair $\p{A,R}$ is a Rota-Baxter algebra of weight $-1$ and we naturally have  $\tilde R\p{f}(x)=f(x)\sigma_-\p{x}$ where $\sigma_-$ is the indicator function of $\R_-$. The same occurs for the pair $\p{\Lin{H,A},\mathcal R}$.

We then set $\dspappli{\theta}{A}{A}{f(x)}{f(-x)}$,  $\sigma:=\theta$, $\gamma:=f_1$ and $\dspappli{\Theta}{A}{\C}{f}{\integ{\R}{}{f(t)}{t}}$.

\begin{proposition}
The $6$-uple $\p{A,R,\theta,\sigma,\Theta,\gamma}$ is an average algebra.
\end{proposition}

\begin{proof}
We verify the six Axioms of Definition \ref{def_wbRBa}
\begin{enumerate}
 \item We have already explained that $\p{A,R}$ is a Rota-Baxter algebra of weight $-1$. 
\item For any $f\in A$, one has, as a consequence of the parity of $f_1^{\star n}$, for any $n\in\N^*$ ~: $$\sigma\p{f_1^{\star n}\star f}= \p{f_1^{\star n}\star f}\p{-x}=\int_{\R} f_1^{\star n}(-x-t)f(t)\text{d}t=\int_\R f_1^{\star n}\p{x+t}f(-t)\text{d}t =\sigma\p{f_1^{\star n}}\star\theta\p{f}.$$ So $\sigma$ satisfies Axiom \ref{def_wbRBa_H2}.  
\item By a simple application of Fubini's Theorem and an affine change of variables in integrals, one verifies that $\Theta$ is an algebra morphism and that $\Theta(f)=1$.
\item For any $f\in A$, one has with $\sigma_{\geq 0}$ the indicator function of $\R_+$ ~: $$\theta\p{\tilde R (f)}=\theta\p{f(x)\sigma_-\p{x}}= f(-x)\sigma_{\geq 0}(x).$$ Thus $\Theta\p{\tilde R(\theta(f))-  R\p{\sigma\p{f}})}=\Theta\p{f(-x)\sigma_{\geq 0}(x)-f(-x)\sigma_+(x)}=0$ and Axiom \ref{def_wbRBa_H3} is fulfilled.
\item One has naturally $\sigma (\gamma)=\gamma$ because $\gamma$ is odd.
\item The Axiom \ref{def_wbRBa_H5} is a straightforward implication of the fact that $\theta=\sigma$, see Remark \ref{remark_simplification2}.
\end{enumerate}
 
\end{proof}

%
%
%

Applying the main theorem \ref{theo_wba}, it comes~:

\begin{theorem}[Diffusion induced averages]

The map $\phidiff:\Hcsh\to\R$ defined for any word $\w\in\Hcsh$ of length $n$ by~:
 \begin{eqnarray*}
 \phidiff\p{\w}&=&(-1)^n\Theta\p{ \p{\mathcal R(a)}^{\cro{n}}\p{\w}}\\
&=& (-1)^n \int_\R \p{\p{\p{\p{f^{*\omega_1}\sigma_+}* f^{*\omega_2}}\sigma_+}*\hdots* f^{*\omega_r}}(x)\sigma_+(x) \text{d}x\\
&=&(-1)^n\begin{split}
\int_{\R^r} f_{\omega_1}\p{x_1}\hdots
f_{\omega_r}\p{x_r} \sigma_{+}\p{\bas{x}_1 
}\hdots
\sigma_{+}\p{\bas{x}_r } \text{d}x_1\hdots \text{d}x_r
\end{split}.
\end{eqnarray*}
 
and extended on the whole of $\Hcsh$ by linearity is solution of the algebraic problem of averages.

Moreover 
\begin{eqnarray*}
 \psi^{*-1}\circ\rev\p{\w}&=&(-1)^n\Theta\p{ \p{\tilde{\mathcal R}(a)}^{\cro{n}}\p{\w}}\\
&=&  (-1)^n\int_\R \p{\p{\p{\p{f^{*\omega_1}\sigma_-}* f^{*\omega_2}}\sigma_-}*\hdots* f^{*\omega_r}}(x)\sigma_-(x) \text{d}x\\
&=&(-1)^n\begin{split}
\int_{\R^r} f_{\omega_1}\p{x_1}\hdots
f_{\omega_r}\p{x_r} \sigma_{-}\p{\bas{x}_1 
}\hdots
\sigma_{-}\p{\bas{x}_r } \text{d}x_1\hdots \text{d}x_r
\end{split}.
\end{eqnarray*}                                                                                                                 

\end{theorem}
    
\begin{proof}
 It is enough to prove that for any $n\in\N$, $\overline{\Theta\p{\p{\mathcal Ra}^{\cro{n]}}}}=\Theta\circ \theta\p{\p{\mathcal Ra}^{\cro{n]}}}$ which comes directly from the parity and of the reality of $f_w$ for any $\omega\in\Omega$.
 
\end{proof}

\begin{WithArbo}

We prove now that $\phidiff$ solves the analytic problem of averages. Using Theorem \ref{theo_wbaa} and $\p{A,R,\theta,\sigma,\Theta,\gamma}$ being an averages algebra,  it is sufficient to find an algebra morphism  $v:A\to A$ verifying axioms of Definition  \ref{def_wbaRB}.

\begin{theorem}
 For $\dspappli{v}{A}{A}{f}{\abs{f}}$, the seven-uple $\p{A,R,\theta,\sigma,\Theta,\gamma,v}$ is a well behaved Rota-Baxter algebra and so $\phidiff$ solves both the algebraic and the analytic problem of averages.
\end{theorem}
\begin{proof}

 For $b\in \im R$ and  $m\in\N$, one has ~:
 \begin{enumerate}
\item $ \abs{\Theta( b)} =\abs{\integ{\R}{}{b(x)}{x}}\leq \integ{\R}{}{\abs{b(x)}}{x}=\Theta\circ v\p{ b} =\abs{\Theta\circ v\p{ b}}$ and so Axiom \ref{axiom_wba1} is fulfilled.
\item Moreover~: 
\begin{align*}
\abs{\Theta\circ v\circ{ R\p{\gamma^m b }}} &= \integ{\R}{}{R\p{\abs{\gamma^m \star b}(x)}}{x}\\
&= \integ{\R}{}{\abs{\gamma^m \star b}(x)\sigma_+(x)}{x}\\
&\leq \int_{\R^2}\abs{\gamma^m(t) b(x-t)}\sigma_+(x)\dd{t}\dd{x}\\
&\leq \int_{\R^2}\abs{\gamma^m(u) b(v)}\sigma_+(u+v)\dd{u}\dd{v}\\
&\leq\int_\R \abs{\gamma^m(u)}\dd{u}\int_{\R}\abs{b(v)}\dd{v}\\
&\leq k^m \abs{\Theta\circ v(  b)} 
\end{align*}
because $\int_\R \abs{\gamma^m(u)}\dd{u}\leq 1$ and with $k=1$ and Axiom \ref{axiom_wba2} is fulfilled as well. 
\end{enumerate}
\end{proof}

Something remarkable is that the character $\phidiff^{\gg}:=\Theta\circ\F_{|t=-1}\circ\rev\circ \alpha$ has a closed form formula.

\begin{proposition}
 With notations of this section and the ones of Section \ref{subsection_Connes_Kreimer_Hopf_algebra}, one has for any $\w^<\in \motsarbo$~:
$$\phidiff^{\gg}\p{\w^<}=\int_{\R^r} f_{\omega_1}\p{x_1}\hdots
f_{\omega_r}\p{x_r} \sigma_{+}\p{\hat{x}_1 
}\hdots
\sigma_{+}\p{\hat{x}_r } \text{d}x_1\hdots \text{d}x_r$$
where $r$ is the number of nodes in the forest $\w^<$ and where the sums $\hat{x_i}$ are related
to the arborescent order\footnote{For example~:\begin{itemize}\item if $\w^<=\scalebox{1.5}{\raisebox{-2mm}{\arbreabcA{\omega_1}{\omega_2}{\omega_3}}}$ then $\hat{x}_1=x_1+x_2+x_3$, $\hat x_2=x_2$ and $\hat x_3=x_3$;\item if $ \scalebox{1.5}{\raisebox{-2mm}{\arbreabcB{\omega_1}{\omega_2}{\omega_3}}}$ then $\hat{x}_1=x_1+x_2$, $\hat x_2=x_2$ and $\hat x_3=x_3$. \end{itemize}}. 
\end{proposition}


\begin{proof}
Our first goal is to prove that for any irreducible arborescent $\w^<=\p{\omega_1,\hdots,\omega_r}^<\in\HCK$, one has
\begin{eqnarray}H^<\p{\w^<}(t_1)=\p{\int_{\R^{r-1}} f_{\omega_1}\p{t_1-\p{t_2+\hdots+t_r}} f_{\omega_2}(t_2) \hdots f_{\omega_{r}}(t_{r})\sigma_+\p{\hat{t}_2}\hdots \sigma_+\p{\hat{t}_{r}}\text{d}t_2\hdots \text{d}t_{r}}\sigma_+(t_1)\label{formula_rec_diff}
\end{eqnarray}
where the $\hat t_i$ are relative to the arborescent order.
If $\w^<=\arbrea{\omega_1}$ then $H^<\p{\w^<}=R(a(\omega_1))=f_{\omega_1}(t)\sigma_+(t)$ and the formula is true if $\l{\w^<}=1$. We assume it is true  for any irreducible arborescent sequences of length $\leq r-1$, 

We consider now an arborescent sequence $\w^<=\scalebox{1.5}{\raisebox{-2mm}{\arbreabcdA{\omega_1}{\w^{1<}}{\vdots}{\w^{l<}}}}
$ of length $r$, then using relation (\ref{equation_R_sur_irred}), it comes~:
$$H^<\p{\w^<}=R\p{a\p{\omega_1} H^<\p{'\w^<} }$$ where $'\w^<=\w^{1<}\oplus\hdots\oplus\w^{l<}$ with $\w^{i<}$ irreducible for any $i\in\intere{1}{l}$. Then using now relation (\ref{equation_R_sur_red}), it comes~:
$$H^<\p{\w^<}=R\p{a\p{\omega_1} \p{H^<\p{\w^{1<}}\hdots H^<\p{\w^{l<}}}}$$ and finally using the definition of the convolution product~:

\begin{eqnarray*}
 H^<\p{\w^<}(t)=\int_{\R^{l-1}} f_{\omega_1}\p{t-t_1}H^<\p{\w^{1<}}\p{t_1-t_2}\hdots
 H^<\p{\w^{l-1<}}\p{t_{l-1}-t_l}H^<\p{\w^{l<}}\p{t_l} \text{d}t_{1}\hdots\text{d}t_l \sigma_+\p{t}
\end{eqnarray*}.

With the change of variables $$\begin{cases} u_1&=t-t_2\\ u_2&=t_2-t_3\\
 \vdots&=\vdots\\u_{l-1}&=t_{l-1}-t_l\\ u_l&=t_l\end{cases}\Longleftrightarrow \begin{cases}t_1&=u_1+\hdots+u_l\\ t_2&=t_1+\hdots+t_{l-1}\\ \vdots&=\vdots\\ t_{l-1}&=u_{l-1}+u_l\\t_l&=u_l\end{cases},$$
one obtains 

\begin{eqnarray*}
 H^<\p{\w^<}(t)=\int_{\R^{l-1}} f_{\omega_1}\p{t-\hat{u}_1}H^<\p{\w^{1<}}\p{u_1}\hdots
 H^<\p{\w^{l-1<}}\p{u_{l-1}}H^<\p{\w^{l<}}\p{u_l} \text{d}u_{1}\hdots\text{d}u_l
\sigma_+\p{t}.
\end{eqnarray*}

Using the induction hypothesis, it comes~:

\begin{eqnarray*}\begin{split}
 H^<\p{\w^<}(t)=\int_{\R^{l-1}} \int_{\R^{r_1-1}}\hdots \int_{\R^{r_l-1}} 
f_{\omega_1}\p{t-\hat{u}_1} \hspace*{8cm}
\\ \cro{
f^<\p{\w^{1<}}\p{u_1,t_2^1,\hdots,t_{r_1}^1}\text{d}t_2^1\hdots \text{d}t_{r_1}^1}
\hdots
\cro{
f^<\p{\w^{l<}}\p{u_l,t_2^l,\hdots,t_{r_l}^l}\text{d}t_2^l\hdots \text{d}t_{r_l}^l}
\text{d}u_{1}\hdots\text{d}u_l
\sigma_+\p{t}\end{split}
\end{eqnarray*}

where $$f^<\p{\w^{i<}}\p{u_i,t_2^i,\hdots,t_{r_i}^i}=f_{\omega_1^i}\p{u_i-\p{t_2^i+\hdots+t_{r_i}^i}} f_{\omega_2^i}(t_2^i) \hdots f_{\omega_{r_i}^i}(t_{r_i})\sigma_+\p{\hat{t}_2^i}\hdots \sigma_+\p{\hat{t}_{r}^i}\sigma_+\p{u_i}.$$

One performs again a change of variable~:

$$\begin{cases}
   T_1&=u_1-\p{t_2^1+\hdots+t_{r_1}^1}\\
\vdots&=\vdots\\
T_l&=u_l-\p{t_2^l+\hdots+t_{r_l}^l}
  \end{cases}$$

and one finds 

\begin{eqnarray*}f^<\p{\w^{i<}}\p{u_i,t_2^i,\hdots,t_{r_i}^i}=f_{\omega_1^i}\p{T_i} f_{\omega_2^i}(t_2^i) \hdots f_{\omega_{r_i}^i}(t_{r_i})\sigma_+\p{\hat{t}_2^i}\hdots \sigma_+\p{\hat{t}_{r}^i}\sigma_+\p{T_i+t_2^i+\hdots+t_{r_i}^i}.\end{eqnarray*}

Finally, 

\begin{eqnarray*}
 H^<\p{\w^<}(t)=\int_{\R^{r-1}} 
f_{\omega_1}\p{t-\p{T_1+t_2^1+\hdots+t_{r_1}^1}-\hdots-\p{T_l+t_2^l+\hdots+t_{rl}^l}}\hspace*{4cm}\\
\hspace*{2cm}\cro{
f^<\p{\w^{1<}}\p{u_1,t_2^1,\hdots,t_{r_1}^1}\text{d}t_2^1\hdots \text{d}t_{r_1}^1}
\hdots
\cro{
f^<\p{\w^{l<}}\p{u_l,t_1^l,\hdots,t_{r_l}^l}\text{d}t_2^l\hdots \text{d}t_{r_l}^l}
\text{d}T_{1}\hdots\text{d}T_l
\sigma_+\p{t}
\end{eqnarray*}
which is of the form we were looking for.

One computes now the formula for $M^{\w\gg}$. For an irreducible arborescent sequence $\w^<$, it suffices to compute the image by $\Theta$ of $H^<\p{\w^<}$. Using Formula (\ref{formula_rec_diff}), one finds~:

\begin{eqnarray*}\begin{split}
&\phidiff^{\gg}\p{\w^<}\\
&=\int_{ \R}  \p{\int_{\R^{r-1}} f_{\omega_1}\p{t-\p{t_2+\hdots+t_r}} f_{\omega_2}(t_2) \hdots f_{\omega_{r}}(t_{r})\sigma_+\p{\hat{t}_2}\hdots \sigma_+\p{\hat{t}_{r}}\text{d}t_2\hdots \text{d}t_{r-1}}\sigma_+(t) \text{d}t\\
&=\int_{ \R^r}  f_{\omega_1}(t_1) \hdots f_{\omega_r}(t_r) \sigma_+\p{\hat{t}_1}\hdots \sigma_+\p{\hat{t}_{r}}\text{d}t_1\hdots \text{d}t_{r-1}
\end{split}
\end{eqnarray*}

by the change of variable $ t_1=t-\hat{t}_2$. The formula for reducible arborescent sequences is  then a consequence of this formula and of the separativity of the character $\phidiff^{\gg}$.

\end{proof}
\end{WithArbo}


\subsection{The organic average}\label{section_organic_average}

In this section, we consider the algebra $A=\R[x,y]$ of complex polynomials in two variables. We define an operator $R$ on $A$ by setting, for any $m,n\in\Omega$, $$R\p{x^m y^n}:=\begin{cases}\dfrac{m}{m+n} x^{m+n}&\text{ if } m+n\neq 0 \\ 1 &\text{ otherwise }\end{cases}$$ and by extending it on the whole of $A$ by linearity.

We then have~:

\begin{proposition}
 The pair $\p{A,R}$ is a Rota-Baxter algebra of weight $-1$.
\end{proposition}

The proof is a straightforward computation.

\begin{remark}
 The complementary projector $\tilde R$ is given by $\tilde R\p{x^m y^n}=x^m y^n-\dfrac{m}{m+n}x^{m+n}$.
\end{remark}

As in section \ref{section_diffusion_induced_averages}, we introduce  $\dspappli{\theta}{A}{A}{P(x,y)}{P(x,x)}$, $\dspappli{\sigma}{A}{A}{P(x,y)}{P(y,x)}$, $\gamma:=xy$, $\dspappli{\Theta}{A}{\R}{P(x,y)}{P(1,1)}$.

We then have~:

\begin{proposition}
The $6$-uple $\p{A,R,\theta,\sigma,\Theta,\gamma}$ is an average algebra.
\end{proposition}

\begin{proof}
We verify the six axioms of Definition \ref{def_wbRBa}
\begin{enumerate}
 \item We have already explain that $\p{A,R}$ is a Rota-Baxter average of weight $-1$. 
\item The map $\sigma$ is an algebra morphism of $A$ and so Axiom \ref{def_wbRBa_H2} is satisfied.
\item For any $x^my^n\in A$, one has~: $$\theta\p{\tilde R \p{x^my^n}}=\theta\p{x^m y^n-\dfrac{m}{m+n}x^{m+n}}= \dfrac{n}{m+n} x^{2n}=R\p{x^ny^m}=R\p{\sigma\p{x^m y^n}}.$$ So for any $P\in A$, $\theta\circ\tilde R(P)=R\circ \sigma (P)$. The equality is true on whole of $A$ so it is yet true modulo $\ker \Theta$.
\item One has evidently $\sigma (\gamma)=\gamma$.
\item In order to verify  Axiom \ref{def_wbRBa_H5}, following Lemma \ref{lemma_axiom5} we compute~:
\begin{eqnarray*}
& & R\cro{\p{xy}^n\p{\sigma\p{\tilde R\p{x^my^n}} - R\p{\sigma\p{x^my^n}}}  }\\
&=&R\cro{\p{xy}^n \p{y^m x^n-\dfrac{m}{m+n}y^{m+n}  - \dfrac{n}{n+m} x^{n+m}}   }\\
&=& R\p{x^{2n}y^{m+n}}-\dfrac{m}{m+n}R\p{x^ny^{m+2n}}-\dfrac{n}{n+m}R\p{x^{2n+m}y^n}\\
&=&\p{ \dfrac{2n}{3n+m} - \dfrac{m}{m+n}\dfrac{n}{m+3n}-\dfrac{n}{n+m}\dfrac{2n+m}{3n+m} }x^{3n+m}\\
&=&0
\end{eqnarray*}
as needed.
\end{enumerate}
 
\end{proof}

%

Applying our main Theorem \ref{theo_wba}, it comes~:

\begin{theorem}[Organic average]

The character $\phiorga$ defined for any non empty words $\w=\p{\omega_1,\hdots,\omega_r}\in\mots$ by~:

 \begin{eqnarray*}
 \phiorga\p{\w}&=&(-1)^n\Theta\p{ \p{\mathcal R(a)}^{\cro{n}}\p{\w}}\\
&=&  (-1)^n\p{\dfrac{\omega_1}{2\omega_1}-1}\p{\dfrac{\omega_2}{2\p{\omega_1+\omega_2}}-1}\hdots \p{\dfrac{\omega_r}{2\p{\omega_1+\omega_2+\hdots+\omega_r}}-1}
\end{eqnarray*}
 is a solution of the algebraic problem of averages.

 \end{theorem}
    
\begin{proof}
 One just needs to prove that $\overline{\Theta\p{\p{\mathcal Ra}^{\cro{n]}}}}=\Theta\circ \theta\p{\p{\mathcal Ra}^{\cro{n]}}}$.  But the equality $\overline\Theta=\Theta\circ\theta$ is true on the whole of $A$. Indeed, for any $P\in A$, one has $$\overline{\Theta(P(x,y))}=P(1,1)=\Theta\p{P(x,x)}= \Theta\circ \theta\p{P(x,y}$$ because $P$ is a real polynomial. 
\end{proof}

\begin{WithArbo}

We prove now that $\phiorga$ solves the analytic problem of averages. 

\begin{theorem}
 For $v=\id_A$, the seven-uple $\p{A,R,\theta,\sigma,\Theta,\gamma,v}$ is a well behaved Rota-Baxter algebra and so $\phiorga$ solves both the algebraic and analytic problem of averages.
\end{theorem}
\begin{proof}
 
 Using remark \ref{remark_restriction_axioms_wba}, we just verify axioms \ref{axiom_wba1} and \ref{axiom_wba2} of definition \ref{def_wbaRB} for elements $b\in\im R\cap \R_+[X]$.  Let us consider thus  $m\in\N$ and $b\in \im R\cap\R_+[X]$. Then $b$ is a polynomial in the variable $x$ with positive coefficients and one has ~:
 \begin{enumerate}
\item By definition of $v$, Axiom \ref{axiom_wba1} is trivially fulfilled.
\item For $m\in\N$ and $b=\sum_{k=0}^{n}a_k x^k \in \im R$, one has~:

\begin{align*}
\abs{\Theta\circ v\circ{ R\p{\gamma^m b }}} &=\abs{\Theta\circ R\p{(xy)^m \p{\sum_{k=0}^{n}a_k x^k}}}\\
&= \abs{\Theta\circ R\p{\sum_{k=0}^{n}a_k x^{k+m}y^m}}\\
&=\abs{\Theta\p{ \sum_{k=0}^{n} \dfrac{k+m}{k+2m}a_k x^{k+2m}}}\\
&=\abs{\sum_{k=0}^{n} \dfrac{k+m}{k+2m}a_k}\\
&\leq {\sum_{k=0}^{n} a_k}\\
&\leq \abs{\Theta(\sum_{k=0}^{n} a_k X^k)}\\
&\leq k \abs{\Theta\circ v(  b)}
\end{align*}
with $k=1$ and Axiom \ref{axiom_wba2} is satisfied as well. 
\end{enumerate}
\end{proof}

As in the previous section, one has a closed formula for  character $\phiorga^{\gg}:=\Theta\circ\F_{|t=-1}\circ\rev\circ \alpha$.

\begin{proposition}
 For any $\w^<\in \HCK$~:
$$\phiorga^{\gg}\p{\w^<}=\prod_{i=1}^r \p{\dfrac{\frac{\omega_i}{2}}{\hat{\omega_i}}-1}
$$
where $r$ denotes the number of nodes of $\w^<$ and where the sums $\hat{x_i}$ are related
to the anti-arborescent order. 
\end{proposition}

\begin{proof}
We will first prove that for any irreducible arborescent sequence $\w^<=\p{\omega_1,\hdots,\omega_r}^<$, one has~:
\begin{eqnarray}
 H^<\p{\w^<}&=& \underbrace{\prod_{i=1}^r \p{\dfrac{\frac{\omega_i}{2}}{\hat{\omega_i}}-1}}_{:=M^{\w\gg}} x^{2\norme{\w^<}}\label{formula_rec_orga}.
\end{eqnarray}

If $\l{\w^<}=1$, then $H^<\p{\w^<}=R(a(\omega_1)=\frac{\omega_1}{2\omega_1}x^{2\omega_1}$ as needed. We assume the formula true for any arborescent sequence of length $\leq r$ and we prove it for an  arborescent sequence $\w^<=\scalebox{1.5}{\raisebox{-2mm}{\arbreabcdA{\omega_1}{\w^{1<}}{\vdots}{\w^{l<}}}}$ of length $r$.

Using relation \ref{equation_R_sur_irred}, it comes~:
$$H^<\p{\w^<}=R\p{a\p{\omega_1} H^<\p{'\w^<} }$$ but one has  $'\w^<=\w^{1<}\oplus\hdots\oplus\w^{l<}$ with $\w^{i<}$ irreducible for each $i\in\intere{1}{l}$ so one has too~:
$$H^<\p{\w^<}=R\p{a\p{\omega_1} \p{H^<\p{\w^{1<}}\hdots H^<\p{\w^{l<}}}}$$
and so using the induction hypothesis~:

\begin{eqnarray*}
 H^<\p{\w^<}&=&R\p{ \p{xy}^{\omega_1} M^{\w^1\gg} x^{2\norme{\w^{1<}}}\hdots  M^{\w^l\gg} x^{2\norme{\w^{l<}}}}\\
&=& \dfrac{\omega_1}{2\p{\omega_1+ \norme{\w^{1<}} + \hdots+ \norme{\w^{l<}}} } M^{\w^1\gg} \hdots M^{\w^l\gg} x^{\norme{\w^<}}
\end{eqnarray*}

which is of the expected form. It remains to apply $\Theta$ to find the aforementioned formula for $\phiorga^{\gg}\p{\w^<}$ with $\w^<$ irreducible. The formula for non irreducible forest follows of the fact that $\phiorga^{\gg}$ is a character.
\end{proof}
\end{WithArbo}

\subsection{The shuffle algebra and the organic average}

One considers here as in subsection \ref{subsection_reminders_quasi_shuffles_algebra} the set $\mots$ of words but now on the alphabet with two letters $\Omega=\ens{1,2}$. We consider on the linear span of $\mots$ the shuffle product (still denoted by $\sh$, there is no possible confusion here with the quasi-shuffle product) and we denote by $\Hsh$ the algebra thus obtained. The empty word $\emptyset$ is the unity of this algebra. 

\begin{theorem}
Let us consider on $\Hsh$ the operator $R$ given by~:
$$R(\w)=\begin{cases} \w &\text{ if the first letter of $\w$ is $1$}\\
0 &\text{ otherwise}  \end{cases}.$$
Then the couple $(\Hsh,R)$ is a unitary commutative Rota-Baxter algebra of weight $-1$.
\end{theorem}

\begin{proof}
It is a direct consequence of characterization of Rota-Baxter algebra of weight $-1$ in terms of two supplementary sub-algebras.
\end{proof}

Let us consider the maps $\sigma:\Hsh\to\Hsh$ defined on a word $\w=\p{\omega_1,\hdots,\omega_r}$ by $\sigma(\w)=\p{\overline\omega_1,\hdots,\overline \omega_r}$ (where $\overline 1=2$ and $\overline 2=1$) and extended on $\Hsh$ by linearity.

We then set $\theta=\sigma$ and $\gamma=\eta+\overline \eta$ where $\eta$ is a fixed non empty word of $\Hsh$ (For example $ \eta=\p{1,2}$ and $\gamma=\p{1,2}+\p{2,1}$).

We finally set  $$\dspappli{\Theta}{\Hsh}{\R}{\w}{\dfrac{1}{\l{\w}!}}.$$ 

\begin{proposition}
The $6$-uple $\p{A,R,\theta,\sigma,\Theta,\gamma}$ is a well-behaved Rota-Baxter algebra.
\end{proposition}

\begin{proof}
We verify the six axioms of Definition \ref{def_wbRBa}
\begin{enumerate}
 \item We have already explained that $\p{A,R}$ is an unitary  Rota-Baxter average of weight $-1$. 
\item The map $\sigma$ is an algebra morphism of $A$ and so Axiom \ref{def_wbRBa_H2} is satisfied.
\item $\Theta$ is the well known exponential character of the shuffle algebra. It is then an algebra morphism and one have $\Theta(\gamma)=1$. 
\item For any word $\w\in\Hsh$, one has $\theta\p{\tilde R(\w)}=R\p{\sigma(\w)}$. This equality is true on the whole of $\Hsh$ and then it is true modulo $\ker \theta$.
\item One has evidently $\sigma (\gamma)=\gamma$.
\item  Axiom \ref{def_wbRBa_H5} is a straightforward consequence of the fact that $\theta=\sigma$, see Remark \ref{remark_simplification2}.
\end{enumerate}
\end{proof}

\begin{lemma} \label{lemma_shuffle_1}$\,$
\begin{itemize}
\item For any $\m=(m_1,\hdots,m_p),\n=(n_1,\hdots,n_q)\in\Hsh$, the number of words in the expansion of $\m\sh \n$ is $\binomial{p+q}{p}$.
\item Moreover, if $m_1=1$ and $n_1=2$, then the number of words beginning by $1$ in the expansion $\m\sh\n$ is $\binomial{p+q-1}{p-1}$
\item  For any words $\w^1,\hdots,\w^p\in\Hsh$ of respective lengths $r_1,\hdots,r_p\in\N$, the sum $\w^1\sh\hdots\sh\w^p$ contains $\binomial{r_1+\hdots+r_p}{r_1,\hdots,r_p}$ words. 
\end{itemize}
\end{lemma}

%
%

\begin{lemma}
For any $\w=\p{\omega_1,\hdots,\omega_r}\in\Hsh$, one has~:
\begin{itemize}
\item $\p{\mathcal R(a)}^{[1]}(\omega_1)$ contains $2^{\omega_1-1}\binomial{l\omega_1}{l,\hdots,l}$ words (all beginning by the letter $1$ and all of length $l\omega_1$),
\item For any $n\in\N^*$ and for any word $\w=(\omega_1,\hdots,\omega_n)\in\Hsh$, if $K(\w)$ denotes the number of words in $\p{\mathcal R(a)}^{[n]}(\w)$ then one has~:
$$K(\w)=\p{1+\dfrac{\norme{\w'}}{\norme{\w}}}\dfrac{2^{\omega_n-1}(l\norme{\w})!}{(l!)^{\omega_n}(l\norme{\w'})!}K(\w') .$$
Moreover, all the words contained in the sum $\p{\mathcal R(a)}^{[n]}(\w)$ are of length $l\norme{\w}$ (and begin by the letter $1$).
\item $K(\w)=\dfrac{2^{\norme{\w}-(n)}(l\norme{\w})!}{(l!)^{\norme{\w}}}\prod_{k=1}^n\p{1+\dfrac{\bas\omega_{k-1}}{\bas\omega_k}}
$.
\end{itemize}
\end{lemma}

\begin{proof}$\,$
\begin{itemize}
\item Using Newton's formula, one has~:
\begin{align*}
\p{\mathcal R(a)}^{[1]}(\omega_1)&=R\p{\gamma^{\omega_1}}\\
&=R\p{ \p{\eta+\overline{\eta}}^{\omega_1}}\\
&=\sum_{k=0}^{\omega_1} \binomial{\omega_1}{k}R\p{
\eta^{\sh k }\sh \overline{\eta}^{\sh \omega_1-k}
}.
\end{align*}
But according with Lemma \ref{lemma_shuffle_1}, the power $\eta^{\sh k}$ contains $\binomial{lk}{l,\hdots,l}=\dfrac{(lk)!}{(l!)^k}$ words all beginning by the letter $1$ and the power $\eta^{\sh \omega_1-k}$ contains $\binomial{l\p{\omega_1-k}}{l,\hdots,l}=\dfrac{(l(\omega_1-k))!}{(l!)^{\omega_1-k}}$ words all beginning by the letter $2$. Then always thanks to Lemma \ref{lemma_shuffle_1}, the product $\eta^{\sh k }\sh \overline{\eta}^{\sh \omega_1-k}$ contains   $\binomial{lk}{l,\hdots,l}\binomial{l\p{\omega_1-k}}{l,\hdots,l}\binomial{l\omega_1-1}{lk-1}$ words beginning by the letter $1$.

So, the number of words beginning by the letter $1$ in  $\p{\mathcal R(a)}^{[1]}(\omega_1)$ is 

\begin{align*}
&\sum_{k=0}^{\omega_1} \binomial{\omega_1}{k}\binomial{lk}{l,\hdots,l}\binomial{l\p{\omega_1-k}}{l,\hdots,l}\binomial{l\omega_1-1}{lk-1}\\
&=\sum_{k=0}^{\omega_1}  
\binomial{\omega_1}{k}
\dfrac{(lk)!}{(l!)^k}
\dfrac{(l(\omega_1-k))!}{(l!)^{\omega_1-k}}
\dfrac{(l\omega_1-1)!}{(lk-1)!(l(\omega_1-k))!}\\
&=\dfrac{ (l\omega_1-1)! }{(l!)^{\omega_1}}
l\sum_{k=0}^{\omega_1 } k\\
&=\dfrac{(l\omega_1-1)! }{(l!)^{\omega_1}} l\omega_1 2^{\omega_1-1}\\
&=2^{\omega_1-1} \dfrac{(l\omega_1)! }{(l!)^{\omega_1}}\\
&=2^{\omega_1-1}\binomial{l\omega_1}{l,\hdots,l}
\end{align*}
 
Let us observe that words in this sum are all of length $l\omega_1$ 
\item Again by Newton's formula~:
\begin{align*}
\p{\mathcal R(a)}^{[n]}(\w)&=R\p{a(\omega_n)\p{\mathcal R(a)^{[n]}(\omega')}}\\
&=R\p{\gamma^{\omega_n}\p{\mathcal R(a)^{[n]}(\omega')}}\\
&=\sum_{k=0}^{\omega_n} \binomial{\omega_n}{k}R\p{\eta^{\sh k}\sh \overline \eta^{\sh \omega_n-k} \sh \p{\mathcal R(a)^{[n]}(\omega')} }.
\end{align*}

The factor $\p{\mathcal R(a)^{[n]}(\omega')}$ contains $K(\w')$ words all beginning by $1$ and of length $l\norme{\omega'}$. 

Let us denote by $f(\omega_n,k)$ (respectively $g(\omega_n,k)$) the number of words beginning by $1$ (respectively $2$) in $\eta^{\sh k}\sh \overline \eta^{\sh \omega_n-k}$. 

Using again Lemma \ref{lemma_shuffle_1}, one has~:
\begin{align*}
f(\omega_n,k)&=
\binomial{kl}{l,\hdots,l}
\binomial{l\p{\omega_n-k}}{l,\hdots,l}
\binomial{l\omega_n-1}{lk-1}\\
&=
\dfrac{(lk)!}{(l!)^k}
\dfrac{(l\p{\omega_n-k})!}{(l!)^{\omega_n-k}}
\dfrac{(l\omega_n-1)!}{(lk-1)! (l(\omega_n-k))!}\\
&=\dfrac{lk (l\omega_n-1)!}{(l!)^{\omega_n}}
\end{align*}
and 
\begin{align*}
g(\omega_n,k)&=
\binomial{kl}{l,\hdots,l}
\binomial{l\p{\omega_n-k}}{l,\hdots,l}
\binomial{l\omega_n-1}{l(\omega_n-k)-1}\\
&=\dfrac{l(\omega_n-k) (l\omega_n-1)!}{(l!)^{\omega_n}}.
\end{align*}

Let us now denote by $K(\w,k)$ the number of words beginning by $1$ in $$\eta^{\sh k}\sh \overline \eta^{\sh \omega_n-k} \sh \p{\mathcal R(a)^{[n]}(\omega')}.$$

Then, always by Lemma \ref{lemma_shuffle_1}, one has~:
\begin{align*}
K(\omega,k)&=
f(\omega_n,k)K(\w')\binomial{l\norme{\w'}+l\omega_n}{l\omega_n}+
g(\omega_n,k)K(\w')\binomial{l\norme{\w'}+l\omega_n-1}{l\omega_n}\\
&=\p{
\dfrac{lk (l\omega_n-1)!}{(l!)^{\omega_n}}
\dfrac{(l\norme{\w})!}{(l\omega_n)!(l\norme{\w'})!}
+
\dfrac{l(\omega_n-k) (l\omega_n-1)!}{(l!)^{\omega_n}}
\dfrac{(l\norme{\w}-1)!}{(l\omega_n)!(l\norme{\w'}-1)!}
}K(\w')\\
&=\p{\dfrac{k}{\omega_n}+\dfrac{(\omega_n-k)}{\omega_n }\dfrac{\norme{\w'}}{\norme{\w}}}\dfrac{(l\norme{\w})!}{(l!)^{\omega_n}(l\norme{\w'})!}K(\w')
\end{align*}
and we then obtain~:
\begin{align*}
K(\w)&={\sum_{k=0}^{\omega_n} \binomial{\omega_n}{k}  K(\w,k)}\\
&=\p{\sum_{k=0}^{\omega_n} \binomial{\omega_n}{k}  \p{\dfrac{k}{\omega_n}+\dfrac{(\omega_n-k)}{\omega_n }\dfrac{\norme{\w'}}{\norme{\w}}}}\dfrac{(l\norme{\w})!}{(l!)^{\omega_n}(l\norme{\w'})!}K(\w')\\
&=
\p{2^{\omega_n-1} + \p{2^{\omega_n}-2^{\omega_n-1}}\dfrac{\norme{\w'}}{\norme{\w}}}
\dfrac{(l\norme{\w})!}{(l!)^{\omega_n}(l\norme{\w'})!}K(\w')\\
&=\p{1+\dfrac{\norme{\w'}}{\norme{\w}}}\dfrac{2^{\omega_n-1}(l\norme{\w})!}{(l!)^{\omega_n}(l\norme{\w'})!}K(\w').
\end{align*}
\item A simple induction leads to 
\begin{align*}
K(\w)&=
\p{1+\dfrac{\bas\omega_{r-1}}{\bas\omega_r}}\hdots
\p{1+\dfrac{\bas\omega_{1}}{\bas\omega_2}}
\dfrac{2^{\hat \omega_2-(n-1)}(l\norme{\w})!}{(l!)^{\hat \omega_2}(l\norme{(\omega_1)})!}K(\omega_1)\\
&=\p{1+\dfrac{\bas\omega_{r-1}}{\bas\omega_r}}\hdots
\p{1+\dfrac{\bas\omega_{1}}{\bas\omega_2}}
\dfrac{2^{\norme{\w}-(n)}(l\norme{\w})!}{(l!)^{\norme{\w}}}
\end{align*}
\end{itemize}
\end{proof}

As a consequence of main Theorem and of what precedes, one has, setting $l=2$~:

\begin{theorem}
For any word $\w=\p{\omega_1,\hdots,\omega_n}\in\Hsh$, the map $\phi=(-1)^n\Theta\circ \p{\mathcal R(a)^{[n]}}$ is a character solution of the algebraic problem of averages. Moreover, for any $\w\in\Hcsh$,  one has 
$$\phi(\w)=(-1)^n\Theta\p{\mathcal R(a)^{[n]}}(\w)= \prod_{i=1}^n \p{\dfrac{\omega_i}{2\bas\omega_i}-1}.$$
\end{theorem}
\begin{proof}$\,$

\begin{itemize}
\item The first point comes directly from the fact that $\Theta$ takes its values in $\R$ and that it just depends on words length, which is preserved by $\Theta$. 
\item A direct consequence of the previous proposition is that~:
\begin{align*}
\phi(\w)
&=\dfrac{(-1)^n}{2^n}\Theta\p{\mathcal R(a)^{[n]}}(\w)\\
&=\dfrac{(-1)^{2n}}{2^n} \prod_{i=1}^n\p{1+\dfrac{\bas\omega_{i-1}}{\bas \omega_i}}\\
&=\prod_{i=1}^n (-1)^n\p{\dfrac{\omega_i}{2\bas\omega_i}-1}
\end{align*}
\end{itemize}

\end{proof}

\begin{theorem}
 For $v=\id_A$, the seven-uple $\p{A,R,\theta,\sigma,\Theta,\gamma,v}$ is a well behaved Rota-Baxter algebra and so $\phi$ solves both the algebraic and analytic problem of averages.
\end{theorem}
\begin{proof}

 \begin{enumerate}
\item By definition of $v$, Axiom \ref{axiom_wba1} of definition \ref{def_wbaRB} is trivially fulfilled.
\item Let us consider $b\in \im R$. It is a sum of words all beginning by the letter $1$ and all of the same length that we denote $l'$. We denote by $n$ the number of words in this sum. For $m\in\N$, $\gamma^m$ it the sum of $\binom{ml}{l,\hdots,l}$ words of length $ml$. Then $\gamma^m b$ is the sum of $\binom{ml+l'}{ml}\binom{ml}{l,\hdots,l}$ words of length $l\p{\norme{\w}+m}$ and 

\begin{align*}
 \abs{\Theta\circ v\circ{ R\p{\gamma^m b }}}&=n\binom{ml+l'}{ml}\binom{ml}{l,\hdots,l}\dfrac{1}{l\p{\norme{\w}+m}!}=\dfrac{1}{(l!)^m}\dfrac{n}{l'!}=\dfrac{1}{(l!)^m}\abs{\Theta(b)}\leq k\abs{\Theta(b)}
\end{align*}

with $k=1$ and Axiom \ref{axiom_wba2} is satisfied as well. 
\end{enumerate}
\end{proof}
\section{Outlook}

We have exposed in these pages how to recover important moulds (which were already discovered by Ecalle) by using classical algebraic tools. Our method is based on the factorization of a very simple character $\iota$ of the quasi-shuffle algebra using Atkinson's recursion but in a constrained way. 

We have explained that performing the same factorization but on the Connes-Kreimer algebra permits to compute the contracted arborification of the  moulds by using the Hopf algebra formalism.

The following step would be naturally to obtain new moulds solving the problem of averages, which remains to find new well behaved Rota-Baxter algebras. As already explained, it is a very difficult problem and the classical Rota-Baxter algebras fail to be well behaved ones. One of the next challenges consists then to find such algebras.

Ecalle's original presentation of averages theory is based on the notion of family of complex weights he assigns to analytic continuations of germs of analytic functions. Such a family of weights must particulary  satisfy some so called \textit{autocoherence relations} (\cite{Ec1}, \cite{Vi1}).  We will explain in a forthcoming paper that it is possible to obtain a very simple algebraic interpretation for the autocoherence relations,  and also to recover another crucial relation present in \cite{MeNoTh} by using Atkinson's recursion on the Hopf algebra of  \textit{packed words} $\WQSym$. 


\frenchspacing
\begin{otherlanguage}{english}

\end{otherlanguage}

\end{document}